\documentclass[reqno, 10pt, oneside]{amsart}

\usepackage{latexsym, amsmath,amssymb}

\usepackage{ulem}

\usepackage{hyperref}
\hypersetup{colorlinks=true, citecolor=cyan, linkcolor=blue}
\usepackage{lmodern} 
\usepackage{textcomp}

\usepackage[initials, alphabetic,backrefs,msc-links]{amsrefs}
\usepackage{graphicx}
\usepackage{xcolor}
\usepackage{enumitem}
\usepackage[utf8]{inputenc}

 \numberwithin{equation}{section}

\usepackage{amsmath}

\newcommand{\strt}[1]{\rule{0pt}{#1}}

\def\XXint#1#2#3{{\setbox0=\hbox{$#1{#2#3}{%
\int}$ }
\vcenter{\hbox{$#2#3$ }}\kern-.6\wd0}}

\renewcommand{\epsilon}{\varepsilon}
\newtheorem{theorem}{Theorem}

\newtheorem{lemma}[theorem]{Lemma}
\newtheorem{corollary}[theorem]{Corollary}

\newtheorem{proposition}[theorem]{Proposition}
\newtheorem{definition}[theorem]{Definition}
\newtheorem{remark}[theorem]{Remark}
\newcommand{\bth}{\begin{theorem}}
\newcommand{\ble}{\begin{lemma}}
\newcommand{\bcor}{\begin{corr}}

\newcommand{\bdeff}{\begin{deff}}

\newcommand{\bprop}{\begin{proposition}}
\newcommand{\ele}{\end{lemma}}
\newcommand{\ecor}{\end{corr}}
\newcommand{\edeff}{\end{deff}}

\numberwithin{theorem}{section}

\newcommand{\eprop}{\end{proposition}}

\newcommand{\eps}{\varepsilon}

\renewcommand{\l}{\lambda}

\renewcommand{\Pi}{\varPi}

\renewcommand{\epsilon}{\varepsilon}

\newcommand{\R}{{\mathbb R}}

\newcommand{\N}{{\mathbb N}}

\newcommand{\norm}[1]{\left\|#1\right\|}

\usepackage{esint}

\usepackage{comment}

\newcommand{\ed}{\color{black}}

\allowdisplaybreaks

\usepackage{geometry}

\def\vint_#1{\mathchoice%
        {\mathop{\kern 0.2em\vrule width 0.6em height 0.69678ex depth -0.58065ex
                \kern -0.8em \intop}\nolimits_{\kern -0.4em#1}}%
        {\mathop{\kern 0.1em\vrule width 0.5em height 0.69678ex depth -0.60387ex
                \kern -0.6em \intop}\nolimits_{#1}}%
        {\mathop{\kern 0.1em\vrule width 0.5em height 0.69678ex
            depth -0.60387ex
                \kern -0.6em \intop}\nolimits_{#1}}%
        {\mathop{\kern 0.1em\vrule width 0.5em height 0.69678ex depth -0.60387ex
                \kern -0.6em \intop}\nolimits_{#1}}}
\def\vintslides_#1{\mathchoice%
        {\mathop{\kern 0.1em\vrule width 0.5em height 0.697ex depth -0.581ex
                \kern -0.6em \intop}\nolimits_{\kern -0.4em#1}}%
        {\mathop{\kern 0.1em\vrule width 0.3em height 0.697ex depth -0.604ex
                \kern -0.4em \intop}\nolimits_{#1}}%
        {\mathop{\kern 0.1em\vrule width 0.3em height 0.697ex depth -0.604ex
                \kern -0.4em \intop}\nolimits_{#1}}%
        {\mathop{\kern 0.1em\vrule width 0.3em height 0.697ex depth -0.604ex
                \kern -0.4em \intop}\nolimits_{#1}}}

\newcommand{\aveint}[2]{\mathchoice%
        {\mathop{\kern 0.2em\vrule width 0.6em height 0.69678ex depth -0.58065ex
                \kern -0.8em \intop}\nolimits_{\kern -0.45em#1}^{#2}}%
        {\mathop{\kern 0.1em\vrule width 0.5em height 0.69678ex depth -0.60387ex
                \kern -0.6em \intop}\nolimits_{#1}^{#2}}%
        {\mathop{\kern 0.1em\vrule width 0.5em height 0.69678ex depth -0.60387ex
                \kern -0.6em \intop}\nolimits_{#1}^{#2}}%
        {\mathop{\kern 0.1em\vrule width 0.5em height 0.69678ex depth -0.60387ex
                \kern -0.6em \intop}\nolimits_{#1}^{#2}}}

\def\XXint#1#2#3{{\setbox0=\hbox{$#1{#2#3}{\int}$}
    \vcenter{\hbox{$#2#3$}}\kern-.5\wd0}}

\newcommand{\vertiii}[1]{{\left\vert\kern-0.25ex\left\vert\kern-0.25ex\left\vert #1 
    \right\vert\kern-0.25ex\right\vert\kern-0.25ex\right\vert}}

\newcommand{\vertii}[1]{{\left\vert\kern-0.25ex\left\vert\kern-0.25ex  #1 
    \kern-0.25ex\right\vert\kern-0.25ex\right\vert}}

\makeatletter
\@namedef{subjclassname@2020}{\textup{2020} Mathematics Subject Classification}
\makeatother

\begin{document}

\title[From generalized Poincaré to Poincaré-Sobolev inequalities]
{From generalized Poincaré to Poincaré-Sobolev inequalities via self-improving methods}

\author[A. Claros]{Alejandro Claros}
\address[A. Claros]{BCAM -- Basque Center for Applied Mathematics, Bilbao, Spain\newline
Universidad del Pa\'is Vasco / Euskal Herriko Unibertsitatea (UPV/EHU), Bilbao, Spain}
\email{aclaros@bcamath.org, aclaros003@ikasle.ehu.eus}

\author[C. P\'erez]{Carlos P\'erez}
\address[C. P\'erez]{BCAM -- Basque Center for Applied Mathematics, Bilbao, Spain\newline
Universidad del Pa\'is Vasco / Euskal Herriko Unibertsitatea (UPV/EHU), Bilbao, Spain\newline
Ikerbasque, Bilbao, Spain}
\email{cperez@bcamath.org}

\author[L. Zheng]{Linfei Zheng}
\address[L. Zheng]{
    Center for Applied Mathematics, Tianjin University, Weijin Road 92, 300072 Tianjin, China		
    \newline
    BCAM -- Basque Center for Applied Mathematics, 48009 Bilbao, Spain}
\email{linfei\_zheng@tju.edu.cn}

\thanks{A. Claros is supported by the Basque Government through the BERC 2022-2025 program, by the Ministry of Science and Innovation through Grant PRE2021-099091 funded by BCAM Severo Ochoa accreditation CEX2021-001142-S/MICIN/AEI/10.13039/501100011033 and by ESF+, and by the project PID2023-146646NB-I00 funded by MICIU/AEI/10.13039/501100011033 and by ESF+. 
\\ \hspace*{1.5em}C. P\'erez is supported by the Spanish government through the grant PID2023-146646NB-I00 and by Severo Ochoa accreditation CEX2021-001142-S, both at BCAM, and also by the Basque Government through grant IT1615-22 at the University of the Basque Country and by the BERC program 2022-2025 at BCAM.
\\ \hspace*{1.5em} L. Zheng is supported by the National Key R \& D Program of China (Grant No. 2021YFA1002500) and China Scholarship Council.}



\begin{abstract}
We establish several improvements to the main results of \cite{PR} and \cite{CP}, refining the seminal self-improving method for generalized Poincaré inequalities from \cite{FPW98, MP98}. These results, together with various related applications, stem from a general self-improving property for functions satisfying the local inequality$$\frac{1}{|Q|}\int_Q |f(x)-f_Q|\,dx \le a(Q)$$for all cubes $Q\subset\mathbb{R}^n$. The functional $a$ is assumed to obey a specific discrete geometric summability condition. By restricting our focus to axis-parallel cubes in $\mathbb{R}^n$, this geometric setting allows us to obtain sharper estimates than those available in more general metric measure spaces.
\end{abstract}

\maketitle

\section{Introduction and main results}

The classical local $(1,1)$ Poincaré inequality asserts that there exists a dimensional constant $c_n>0$ such that for every $f\in W^{1,1}_{\operatorname{loc}}(\R^n)$ and every cube $Q\subset \R^n$ one has 
\begin{equation}\label{I1 P11}
	\frac{1}{|Q|}\int_Q |f(x)-f_Q|\,dx \le c_n \,\ell(Q)\,\frac{1}{|Q|}\int_Q |\nabla f(x)|\,dx, 
\end{equation} 
where $f_Q=\frac{1}{|Q|}\int_Q f(y)\,dy$ denotes the average of $f$ over $Q$, and $\ell(Q)$ is the side length of $Q$.  

This inequality enjoys a self-improving property that has been the focus of recent studies. Notably, this phenomenon does not stem from the presence of the gradient on the right-hand side of \eqref{I1 P11}, but rather is a consequence of a discrete summation condition inherent to the functional therein. Motivated by this, we consider inequalities of the form 
\begin{equation}\label{I2}
	\frac{1}{|Q|}\int_Q |f(x)-f_Q|\,dx \le a(Q),
\end{equation}
valid for every cube $Q$, where $a$ is a functional defined on the family of all cubes in $\R^n$ with sides parallel to the coordinate axes. 
 Inequalities of this type, featuring a local oscillation term on the left-hand side and a functional $a(Q)$ on the right, will be referred to as generalized Poincaré inequalities.

Generalized Poincaré inequalities and their self-improving properties were first studied in a general framework by Franchi, the second author, 
and Wheeden \cite{FPW98}; see also \cite{MP98, LernerPerez}  for further refinements. 
In this setting,   the functional $a$ is required to satisfy a suitable geometric discrete summation condition.
This approach provides a way to obtain weighted Poincaré--Sobolev inequalities without relying on classical representation formulas.

A major motivation for studying weighted Poincaré--Sobolev inequalities comes from the regularity theory of degenerate elliptic equations. In their seminal work, Fabes, Kenig, and Serapioni \cite{FKS} used such inequalities as part of the Moser iteration scheme to prove local Hölder regularity for weak solutions of degenerate elliptic equations whose ellipticity is controlled by an $A_2$ weight (see also \cite{HKM}).

While the self-improving results in these earlier works hold in the more general setting of spaces of homogeneous type, the specific geometry of cubes in $\mathbb{R}^n$ (with sides parallel to the axes) allows for a more delicate analysis. In the present work, we exploit this framework to obtain more precise estimates than those treated in \cite{FPW98} or \cite{MP98}. Within this context, the discrete summation condition introduced in \cite{FPW98} can be stated as follows.

\begin{definition}\label{DefDp}
	Let $w$ be any weight in $\R^n$. We say that the functional $a$ satisfies the weighted $D_p(w)$ condition for $0< p<\infty$ if there is a constant $C$ such that, for any cube $Q$ and any family $\{Q_j\}_j$ of pairwise disjoint subcubes of $Q$, the following inequality holds:
		\begin{equation*}
			\sum_j a(Q_j)^p w(Q_j)  \le C^p  a(Q)^p w(Q) ,
		\end{equation*}
		where $w(E)=\int_E w(x)dx$. The best possible constant $C$ above is denoted by $\left\| a \right\|_{D_p(w)} $, and in this case we write $a\in  D_p(w)$.
\end{definition}

In \cite{FPW98} it was proved that if $w\in A_\infty$ and $a\in D_p(w)$, then the starting point inequality \eqref{I2} implies a weighted weak $L^p$ generalized Poincaré inequality. More recently, the dependence of the constants in the main result was improved in \cite{CP}, always in the context of cubes as a consequence of an extension of the John--Nirenberg theorem. For completeness, we state the latest version of this result below.  

\begin{theorem}[\cite{FPW98,CP}] \label{Thm 1.6 CP}
		Let $w\in A_\infty$ and let $1<p<\infty$. Consider a functional $a$ satisfying the $D_p(w)$ condition with constant $\left\| a \right\|_{D_p(w)} $. Let $f$ be a locally integrable function such that
		\begin{equation*}
			\frac{1}{|Q|} \int_Q |f(x)-f_Q|dx \le a(Q)
		\end{equation*}
		for every cube $Q$. Then, for every cube $Q$,
		\begin{equation}\label{I3}
			\left\| f-f_Q\right\|_{ L^{p,\infty} \left( Q , \frac{w(x)dx}{w(Q)}\right) } \le c_n p  [w]_{A_\infty} \left\| a \right\|_{D_p(w)}  a(Q).
		\end{equation}
\end{theorem}

As a consequence of the so-called Kolmogorov inequality, \eqref{I3} yields the following strong estimate:
\begin{equation*}
	\left( \frac{1}{w(Q)} \int_Q |f(x)-f_Q|^{q} w(x)\,dx\right)^{\frac{1}{q}} 
	\le c_n \left( \frac{p}{p-q}\right)^{\frac{1}{q}} p [w]_{A_\infty} \left\| a \right\|_{D_p(w)} \, a(Q),
\end{equation*}
for all $0<q<p$. Thus, the $D_p(w)$ condition implies that \eqref{I2} improves to a weighted $L^q$ inequality for every $0<q<p$, although in general it does not yield the strong norm $L^p$. 
In the case of functionals for which the truncation method (or a weak-implies-strong argument) is applicable, \eqref{I3} does indeed imply the $L^p$ estimate (see \cite{KO03}).  More recently, the condition introduced in Definition \ref{DefDp} was sharpened in \cite{PR}. 
This more refined formulation is satisfied by the fundamental examples and was specifically designed to yield direct strong-type $L^p$ estimates, in contrast to the weak-type $L^p$ bounds obtained under the original $D_p$ condition. The new condition, denoted by $SD_p^s(w)$, is stronger than $D_p(w)$, yet it is satisfied by a broad class of functionals; see \cite{PR, CMPR, HMPV} for examples in a variety of settings.  	
The precise definition is the following. 

\begin{definition}
	Let $w$ be any weight in $\R^n$, and let $s>0$. We say that the functional $a$ satisfies the weighted $SD^s_p(w)$ condition for $0< p<\infty$ if there is a constant $C$ such that, for any cube $Q$ and any family $\{Q_j\}$ of pairwise disjoint subcubes of $Q$, the following inequality holds:
		\begin{equation*}
			\sum_j a(Q_j)^p w(Q_j)  \le C^p 
			\left(\frac{\left| \bigcup_j Q_j \right| }{|Q|} \right) ^{\frac{p}{s}}  a(Q)^p w(Q) .
		\end{equation*}
	The best possible constant $C$ above is denoted by $\left\| a \right\|_{SD_p^s(w)} $, and in this case we write $a\in  SD^s_p(w)$.
\end{definition}

The constants $\|a\|_{D_p(w)}$ and $\|a\|_{SD_p^s(w)}$ are always greater than or equal to 1; we shall omit the subscript when there is no ambiguity; moreover, when working with Lebesgue measure (i.e.\ $w=1$) we will also drop the dependence on $w$ and simply write $\|a\|_{D_p}$ and $\|a\|_{SD_p^s}$. This condition may be viewed as a strengthened version of the $D_p(w)$ condition, incorporating an additional smallness preservation factor on the right-hand side. A typical example of a functional satisfying this condition is  
\begin{equation*}
	a(Q)=\ell(Q)^\alpha \left(\frac{\nu(Q)}{w(Q)}\right)^{\frac{1}{p}},
\end{equation*}  
where $\nu$ is a locally finite measure and $\alpha, p>0$. It was shown in \cite[Lemma 3.2]{PR} that $a \in SD_p^{n/\alpha}(w)$. 
	  
The main self-improving result of \cite{PR} states that if $w\in A_\infty$ and $a\in SD^s_p(w)$, then the inequality \eqref{I2} implies a weighted $L^p$ generalized Poincaré inequality. In that paper,  the second author 
and Rela conjectured that the $A_\infty$ assumption could be removed. A partial step in this direction was later obtained by Martínez-Perales \cite{JM}. More recently, the conjecture was fully resolved by Lerner, Lorist, and Ombrosi \cite{LLO}. Their approach relies on sparse domination to derive a pointwise estimate and, moreover, yields an improved dependence on $s$ in the resulting constant.  One may ask if the assumption $w \in A_\infty$ in the $D_p$-type self-improving result, Theorem \ref{Thm 1.6 CP}, can also be removed; the answer is no. We construct a counterexample in Proposition \ref{prop:Ainfty-necessity-CP} to illustrate this phenomenon.

\begin{theorem} [\cite{PR,LLO}] \label{Thm 5.3 LLO}
		Let $1<p<\infty$. Consider a functional $a$ satisfying the $SD^s_p(w)$ condition with $s>0$ and constant $\left\| a \right\| $. Let $f$ be a locally integrable function such that
		\begin{equation*}
			\frac{1}{|Q|} \int_Q |f(x)-f_Q|dx \le a(Q)
		\end{equation*}
		for every cube $Q$. Then there exists a dimensional constant $c_n$ such that, for any cube $Q$,
		\begin{equation}{\label{results shown in PR and LLO}}
			\left( \frac{1}{w(Q)} \int_Q |f(x)-f_Q|^{p} w(x)dx\right) ^\frac{1}{p} \le c_n (1+s) \left\| a \right\|  a(Q).
		\end{equation}
\end{theorem}
The $SD_p^s(w)$ condition is notably flexible and has found applications in several contexts, including the multi-parameter setting of rectangles \cite{CMPR} and fractional Poincaré--Sobolev inequalities \cite{HMPV}.

\subsection{Main self-improving results}

In this work, we improve the self-improving result under the assumption $a \in SD_p^s(w)$ and $w \in A_r$ for some $1 \le r < \infty$. Moreover, we show that our result is sharp for a wide range of $r$.

\begin{theorem}\label{Selfimproving Ar}
	Let $1\le p<\infty$, $1\le r<\infty$, $1\le s<\infty$, and let $a$ be a functional over cubes. Let $w\in A_r$ and let $f$ be a locally integrable function satisfying
	\begin{equation}\label{starting-point0}
		\frac{1}{|Q|}\int_Q |f(x)-f_Q|\,dx \le a(Q)
	\end{equation}
	for every cube $Q$.
	
	\begin{itemize}
		\item Assume that $a\in SD_p^s(w)$ and $p<rs$. Then there exists a dimensional constant $c_n>0$ such that, for every cube $Q$,
		\begin{equation}\label{Thm Ar eq}
			\norm{f-f_Q}_{\strt{2.3ex} L^{p_{r,s}^{*},\infty}\left(Q, \frac{w(x)\,dx}{w(Q)}\right)}
			\le
			c_n p_{r,s}^{*}\,[w]_{A_\infty}\,[w]_{A_r}^{\frac{1}{rs}}\,
			\|a\|\,a(Q),
		\end{equation}
		where $p_{r,s}^{*}$ is defined by the relation
		\begin{equation*}
			\frac{1}{p}-\frac{1}{p_{r,s}^{*}}=\frac{1}{s\, r}.
		\end{equation*}
		
		\item Assume that $a\in SD_p^s(w)$ and $p\ge rs$. Then there exists a constant $c_{n,p}>0$ such that, for every cube $Q$,
		\begin{equation}\label{Thm Ar eq p>s}
			\|f-f_Q\|_{\exp L\left(Q,\frac{w(x)\,dx}{w(Q)}\right)}
			\le
			c_{n,p}\,[w]_{A_\infty}\,[w]_{A_r}^{\frac1p}\,
			\|a\|\,a(Q).
		\end{equation}
	\end{itemize}
\end{theorem}

Here $\|\cdot\|_{\exp L(X,\mu)}$ denotes the usual Luxemburg norm associated to the Young function $\Phi(t)=e^{t}-1$, namely
\begin{equation*}
\|f\|_{\exp L(X,\mu)} :=\inf\Big\{\lambda>0:\ \int_X \Phi\!\left(\frac{|f(x)|}{\lambda}\right)\,d\mu(x)\le 1\Big\}. 
\end{equation*}

Estimate (\ref{Thm Ar eq}) reflects the compatibility between the $D_p$ condition and the $SD_p^s$ condition. Using H\"older's inequality, we know that $D_{p^*_s} \subseteq SD_p^{s} \subseteq SD_1^{(p^{*}_s)'  }$ with $$\|a\|_{SD_1^{(p^{*}_s)'  }} \le \|a\|_{SD_p^{s}} \le \|a\|_{D_{p^*_s}},$$ 
where $1 \le p < s$ and $p^*_s=\frac{sp}{s-p}$. Applying estimate (\ref{Thm Ar eq}) with $r=1$, one can improve the self-improving result under the condition $a \in SD_p^s$ to the sharp result of the $D_{p^*_s}$ condition, namely the weak-$L^{p^*_s}$ generalized Poincar\'e inequality. Comparing (\ref{Thm Ar eq}) with the estimate ({\ref{results shown in PR and LLO}}) presented in \cite{PR,LLO}, we significantly improve their result when $w$ is an $A_r$ weight.

Estimate \eqref{Thm Ar eq p>s} extends \cite[Corollary 1.9]{PR} in two directions: we do not assume any monotonicity of the functional $a$, and we work in the weighted setting under the more general condition $a\in SD_p^s(w)$.

\begin{remark}
Consider the unweighted setting $w= 1$, so that $w\in A_1$ and we may take $r=1$ in Theorem \ref{Selfimproving Ar}. In the classical Poincaré inequality, we have 
\begin{equation*}
	\frac1{|Q|}\int_Q |f(x)-f_Q|dx \le c_n  \ell(Q)\left(\frac{1}{|Q|}\int_Q |\nabla f(x)|^pdx\right)^{\frac{1}{p}},
\end{equation*}
and the corresponding functional belongs to $SD^{n}_p$; thus $s=n$ is the natural ``dimension'' in this toy model situation. Then for $1\le p<s=n$, the exponent $p_{1,n}^{*}$ in \eqref{Thm Ar eq} becomes the classical critical Sobolev exponent $p^*=\frac{np}{n-p}$. On the other hand, in the borderline and supercritical regime $p\ge n$ (equivalently $p\ge rs$ with $r=1$ and $s=n$), Theorem \ref{Selfimproving Ar} gives us the exponential endpoint \eqref{Thm Ar eq p>s}, which is consistent with the expected classical Trudinger-type behavior when $p\ge n$. 
\end{remark}

 \begin{remark}\label{rem:sharpness-Ar}
The exponent $p_{r,s}^{*}$ in \eqref{Thm Ar eq} is also sharp in the range $1 \le r \le p$ and $s \ge n$. Indeed, in Section \ref{good lambda} we construct weights $w\in A_r$ and smooth test functions for which the conclusion of \eqref{Thm Ar eq} fails if one replaces $p_{r,s}^{*}$ by any larger exponent in those cases; see Proposition~\ref{counterexample proposition} and Proposition~\ref{prop: counterexample fractional p=1}.
\end{remark}

In Appendix \ref{S Polynomials}, we extend Theorem \ref{Selfimproving Ar} to higher-order oscillations by replacing the average $f_Q$, in both the hypothesis and the conclusion, with an appropriate polynomial $P_Q^m f$. This polynomial denotes the projection of $f$ onto the space of polynomials of degree $m$ in $n$ variables over $Q$, leading to corollaries of Poincaré-type inequalities involving higher-order derivatives.  In Appendix \ref{S vector}, we also extend the results of Theorem \ref{Selfimproving Ar} to vector-valued generalized Poincaré inequalities. Specifically, we consider functions $f: \mathbb{R}^n \longrightarrow \ell_q$, where the oscillation in both the hypothesis and the conclusion is measured with respect to the $\ell_q$-norm, appearing as $\norm{f-f_Q}_{\ell_q}$. Finally, in Appendix \ref{S Rectangles}, we extend the result to rectangles in $\R^n$.  

For the $A_\infty$ class of weights, we can prove the following result. 

\begin{theorem}\label{Selfimproving Ainfty}
	Let $w\in A_\infty$, let $1<p<\infty$ and let $a$ be a functional over cubes. Let $f:\R^n \longrightarrow \R$ be a locally integrable function such that
		\begin{equation}\label{starting-point}
			\frac{1}{|Q|} \int_Q |f(x)-f_Q|dx \le a(Q)
		\end{equation}
		for every cube $Q$.
		If $a\in SD_p^s(w)$ for $s>0$, then for every cube $Q$,
			\begin{equation*}
				\norm{f-f_Q}_{\strt{2.3ex}L^{p\left(1+\frac{1}{s}\right),\infty}\left(Q, \frac{w(x)dx}{w(Q)}\right)} \le c_n p \left(1+\frac{1}{s}\right) [w]_{A_\infty} \left\| a\right\|^\frac{s}{s+1} a(Q).
			\end{equation*}	
\end{theorem}

\begin{remark}
	Formally, the condition $SD_p^s(w)$ with $s=\infty$ coincides with $D_p(w)$. Our result recovers the sharp $D_p$ type self-improving result Theorem \ref{Thm 1.6 CP}, with the same constant.
\end{remark}

Comparing with the estimate (\ref{Thm Ar eq}), Theorem \ref{Selfimproving Ainfty} provides a better result under the assumption $ r > p\left(1+\frac{1}{s}\right)$, in the sense that the norm on the left-hand side is larger. 

We also provide a self-improving result for {\it general} weights. 

\begin{theorem}\label{Thm1}
	Let $0< p <\infty$ and let $w$ be any weight. Consider a functional $a$ satisfying the $SD^s_p(w)$ condition with $s>0$ and constant $\left\| a \right\| $. Let $f$ be a locally integrable function such that the following two conditions are satisfied: 
		\begin{equation}\label{Thm1 1}
			\frac{1}{|Q|} \int_Q |f(x)-f_Q|dx \le a(Q)
		\end{equation}
		and,
		\begin{equation}\label{Thm1 2}
			\left\| f - f_Q\right\| _{\strt{2.3ex} L^{p,\infty}\left( Q, \frac{w(x)dx}{w(Q)}\right) } \le K a(Q),
		\end{equation}
		for every cube $Q$, where $K$ is a constant independent of $Q$. Then, there exists a dimensional constant $c_n>0$ such that for any cube $Q$, 
		\begin{equation*} 
			\left\| f - f_Q\right\| _{\strt{2.3ex} L^{p\left(1+\frac{1}{s}\right),\infty}\left( Q, \frac{w(x)dx}{w(Q)}\right) } \le c_n  \max \left\lbrace 1, K \left\| a \right\| \right\rbrace ^\frac{s}{s+1} a(Q).
		\end{equation*}
\end{theorem}

Condition \eqref{Thm1 2} is fairly mild and holds for a wide range of practical cases. Indeed, there are at least two standard ways to verify it. First, if $w\in A_\infty$, one may invoke Theorem \ref{Thm 1.6 CP} to obtain a weighted weak-$L^p$ estimate. This yields a conclusion of the same qualitative nature as Theorem \ref{Selfimproving Ainfty}, albeit with a better dependence on $[w]_{A_\infty}$ (and, in contrast, a worse dependence on $\|a\|$). Second, for completely general weights, one may appeal to Theorem \ref{Thm 5.3 LLO}, which provides a weighted strong $L^p$ generalized Poincaré inequality. This immediately implies \eqref{Thm1 2} with $K = c_n(1+s)\|a\|$. Therefore, Theorem \ref{Thm1} applies under very mild additional input, and as an immediate corollary, we recover and sharpen the main result of~\cite{PR} and its recent refinement in~\cite{LLO}.

\subsection{Weighted Poincaré--Sobolev inequalities}\label{Section PS}

As a first application of our main self-improving result, Theorem \ref{Selfimproving Ar}, we recover the following weighted weak Poincaré-Sobolev estimate: for every $w\in A_r$ with $1\le r\le p$ and $1\le p <nr$ we have 
\begin{equation}\label{self improving example 1}
	\norm{f-f_Q}_{\strt{2.3ex} L^{p_r^{*},\infty}\left(Q, \frac{w(x)\,dx}{w(Q)}\right)} \le c_n p^*_r [w]_{A_\infty}\,[w]_{A_p}^{\frac{1}{p}}\,[w]_{A_r}^{\frac{1}{rn}}\, \ell(Q) \left( \frac{1}{w(Q)}\int_{Q}|\nabla f(x)|^p\, w(x)\,dx\right)^{\frac{1}{p}},
\end{equation}
for every cube $Q$, where $p^*_r$ is defined by the relation
\begin{equation}\label{eq:p*r Section 1.2}
	\frac{1}{p}-\frac{1}{p_r^{*}}=\frac{1}{n }\frac{1}{r}.
\end{equation}

\begin{remark}
The exponent $p_r^{*}$ on the left-hand side of \eqref{self improving example 1} is sharp within the class $A_r$. We prove this in Proposition \ref{counterexample proposition} below; see also \cite[Theorem 2.17]{ClarosJFA} for related sharpness results.
\end{remark}

Furthermore, the weak norm in the left-hand side of \eqref{self improving example 1} can be strengthened to a Lorentz norm by means of a standard truncation argument (see Section \ref{Section 2}).

\begin{corollary}\label{thm:PS-Lorentz-strong} 
Let $w\in A_r$ with $1\le r\le p$, $1\le p< nr$ and consider $p_r^{*}$ defined in \eqref{eq:p*r Section 1.2}. Then, there exists a dimensional constant $c_n>0$ such that 
\begin{equation}\label{eq:PS-Lorentz-strong}
	\|f-f_{Q}\|_{\strt{2.3ex}L^{p_r^{*},p}\left(Q, \frac{w(x)\,dx}{w(Q)}\right)} \le c_n p^*_r [w]_{A_\infty}\,[w]_{A_p}^{\frac{1}{p}}\,[w]_{A_r}^{\frac{1}{rn}}\, \ell(Q) \left( \frac{1}{w(Q)}\int_{Q}|\nabla f(x)|^p\, w(x)\,dx\right)^{\frac{1}{p}},
\end{equation}
for every cube $Q$. 
\end{corollary}

\begin{remark}{\label{remark of the sharpness on Lorentz}}
The second Lorentz index $p$ in the target space $L^{p_r^{*},p}$ in \eqref{eq:PS-Lorentz-strong} is natural from the viewpoint of Sobolev-Lorentz embeddings. Already in the unweighted case $w=1$ (hence $r=1$ and $p_r^{*}=p^*$), the endpoint embedding $W^{1,p}\hookrightarrow L^{p^*,p}$ is optimal along the Lorentz scale (and, more generally, within rearrangement-invariant spaces), in the sense that $W^{1,p}\not\hookrightarrow L^{p^*,q}$ for any $q<p$ (see \cite{EKP}). Consequently, we should not expect a general improvement of \eqref{eq:PS-Lorentz-strong} with $L^{p_r^{*},q}$, $q<p$, even in the simplest situation.
\end{remark}

Furthermore, the $L^{p_r^*,\infty}$ norm on the left-hand side of \eqref{self improving example 1}  can be maintained while sharpening the right-hand side; this is achieved by replacing the strong $L^p$ norm of the gradient with the smaller Lorentz quasi-norm $L^{p,p_r^*}$. This improvement is made precise in the next result.

\begin{corollary}\label{thm:PS-Lorentz-weak}
Let $w\in A_r$ with $1\le r\le p$, $1  < p< nr$ and consider $p_r^{*}$ defined in \eqref{eq:p*r Section 1.2}. Then, there exists a dimensional constant $c_n>0$ such that 
\begin{equation}\label{eq:PS-Lorentz-weak}
	\|f-f_Q\|_{\strt{2.3ex} L^{p_r^{*},\infty}\left(Q, \frac{w(x)\,dx}{w(Q)}\right)} \le c_n p^*_r [w]_{A_\infty}\,[w]_{A_{p,p_r^{*}}}^{\frac{1}{p}}\,[w]_{A_r}^{\frac{1}{rn}}\, \ell(Q) \,\|\nabla f\|_{L^{p,p_r^{*}}\left(Q,\frac{w(x)\,dx}{w(Q)}\right)} ,
\end{equation}
for every cube $Q$, where $[w]_{A_{p,p_r^{*}}}$ denotes the characteristic constant associated with the class $A_{p,p_r^{*}}$ of Chung, Hunt, and Kurtz, as defined in Definition \ref{def:CHK}. 
\end{corollary}

\begin{remark}
The truncation argument used to strengthen weak-type estimates into Lorentz space bounds (as in Corollary \ref{thm:PS-Lorentz-strong}) is intrinsically related to a summability condition on the right-hand side: it is applicable provided that the Lorentz index satisfies $q \le p$ (see Theorem \ref{thm:truncation-method}). This is consistent with the $L^p$ norm of the gradient term in \eqref{self improving example 1}, but it does not apply to the smaller Lorentz norm $\|\nabla f\|_{L^{p,p_r^{*}}}$ in \eqref{eq:PS-Lorentz-weak}, since $p_r^{*}>p$. Consequently, the same argument cannot strengthen \eqref{eq:PS-Lorentz-weak} to an $L^{p_r^{*},p}$ estimate on the left-hand side. 
\end{remark}

\begin{remark}
We observe that for $w \in A_r$ with $1 < r \le p$, the preceding arguments in conjunction with the precise open property of the $A_r$ class (cf. \cite{HPR, IPER}) allow for a refinement of the above estimates. In particular, the target Sobolev exponent can be extended beyond $p_r^*$ to $p_{r-\varepsilon}^*$ for some $\varepsilon = \varepsilon(n,r,w) > 0$, following the approach in \cite{ClarosJFA}. We do not explore this direction further in the present work.
\end{remark}

In the critical and supercritical range $p\ge rn$, the Sobolev exponent $p_r^{*}$ is no longer finite, so the $L^q$ estimate must be replaced by an endpoint estimate of exponential type. The next result provides such an $\exp L$ Poincaré-Sobolev inequality. 

\begin{corollary}\label{thm:PS-exp}
Let $w\in A_r$, let $1<p<\infty$, and assume that $p\ge nr$. Then, there exists a constant $c_{n,p}>0$ such that
\begin{equation}\label{eq:PS-exp}
	\|f-f_Q\|_{\exp L\left(Q, \frac{w(x)\,dx}{w(Q)}\right)} \le c_{n,p} [w]_{A_\infty}\,[w]_{A_{p,q}}^{\frac{1}{p}}\,[w]_{A_r}^{\frac{1}{p}}\, \ell(Q)\,\|\nabla f\|_{L^{p,q}\left(Q,\frac{w(x)\,dx}{w(Q)}\right)},
\end{equation}
for every cube $Q$ and every $1 \le q<\infty$.  
\end{corollary}

We finish this section with two results in different settings. We have the following analogue of Corollary \ref{thm:PS-Lorentz-strong} for pairs of weights (see Definition \ref{def:pair}).

\begin{corollary}\label{prop:pairs}  
	Let $(u,v)\in A_p$ and $u\in A_r$ with $1\le r\le p$. Let $p_r^{*}$ be defined in \eqref{eq:p*r Section 1.2}.
	
	\begin{itemize}
		\item Assume that $1\le p < nr$. Then there exists a dimensional constant $c_n>0$ such that, for every cube $Q$,
		\begin{equation*}
			\|f-f_{Q}\|_{\strt{2.3ex} L^{p_r^{*},p}\left(Q, \frac{u(x)\,dx}{u(Q)}\right)} \le c_n p^*_r [u]_{A_\infty}\, [u]_{A_r}^{\frac{1}{rn}}\, [u,v]_{A_p}^{\frac{1}{p}}\, \ell(Q) \left( \frac{1}{u(Q)}\int_{Q}|\nabla f(x)|^p\, v(x)\,dx\right)^{\frac{1}{p}}.
		\end{equation*}
		
		\item Assume that $p \ge nr$. Then there exists a constant $c_{n,p}>0$ such that, for every cube $Q$,
		\begin{equation*}
			\|f-f_Q\|_{\exp L\left(Q, \frac{u(x)\,dx}{u(Q)}\right)} \le c_{n,p} [u]_{A_\infty}\,[u]_{A_r}^{\frac{1}{p}}\,[u,v]_{A_p}^{\frac{1}{p}}\, \ell(Q) \left( \frac{1}{u(Q)}\int_{Q}|\nabla f(x)|^p\, v(x)\,dx\right)^{\frac{1}{p}}.
		\end{equation*}
	\end{itemize}
\end{corollary} 
\

\subsection{Fractional Poincaré-Sobolev inequalities}\label{Section FPS}

In this section, we establish weighted Poincaré-Sobolev inequalities whose right-hand side is given by the local fractional Sobolev seminorm (also called Gagliardo seminorm). Let $1\le p<\infty$ and $0<\delta<1$. A classical fractional Poincaré-Sobolev inequality was obtained by Bourgain, Brezis, and Mironescu \cite{BBM1} (see also \cite{BBM2}), which shows that there exists a dimensional constant $c_n>0$ such that, for every cube $Q\subset\mathbb{R}^n$,
\begin{equation}\label{PSF SP}
	\frac{1}{|Q|}\int_Q |f(x)-f_Q|\,dx \le c_n\,(1-\delta)^{\frac{1}{p}}\,\ell(Q)^\delta \left(\frac{1}{|Q|}\int_Q\int_Q \frac{|f(x)-f(y)|^p}{|x-y|^{n+\delta p}}\,dy\,dx\right)^{\frac{1}{p}}.
\end{equation}
For fixed $p$ and $0<\delta<1$, the right-hand side of \eqref{PSF SP} is smaller than $\ell(Q) \|\nabla f\|_{L^p(Q, \frac{dx}{|Q|})}$ (see for instance \cite{Inverse}). The BBM extra gain $(1-\delta)^{1/p}$ here reflects the correct scaling of the fractional seminorm: in
\cite{BBM2} it is proved that there exists $K_{n,p}>0$ such that for every $f\in W^{1,p}(Q)$,
\begin{equation*}
	\lim_{\delta\to 1^-} (1-\delta)^{\frac{1}{p}} \left(\frac{1}{|Q|}\int_Q\int_Q \frac{|f(x)-f(y)|^p}{|x-y|^{n+\delta p}}\,dy\,dx\right)^{\frac{1}{p}} = K_{n,p}\left(\frac{1}{|Q|}\int_Q |\nabla f(x)|^p\,dx\right)^{\frac{1}{p}}.
\end{equation*}

We will use as a starting point a weighted version of \eqref{PSF SP}. More
precisely, we shall rely on a weighted estimate proved in \cite{MPW} (see
Theorem \ref{thm:MPW23} below), which is a weighted counterpart of \eqref{PSF SP} for $A_p$ weights. Combining this starting inequality with Theorem \ref{Selfimproving Ar}, we obtain fractional Poincaré-Sobolev inequalities and, in the critical regime, exponential integrability. The main result of this section is the following
corollary.

\begin{corollary}\label{thm:fractional-PS}
	Let $0<\delta<1$ and $w\in A_r$ with $1\le r\le p$. 

	\begin{itemize}
		\item Assume that $1\le p < \frac{nr}{\delta}$. Then there exists $c_{n}>0$ such that, for every cube $Q$,
		\begin{align*}
			\left\| f-f_{Q}\right\|_{\strt{2.3ex} L^{p^*_{\delta, r}, p} \left(Q, \frac{w(x)dx}{w(Q)} \right)} 
			\le & \, c_n p^*_{\delta, r} [w]_{A_\infty} [w]_{A_p}^{\frac{1}{p}}[w]_{A_r}^{\frac{\delta }{rn}}(1-\delta)^\frac{1}{p} \cdot  \\
			& \cdot \ell(Q)^\delta \left( \frac{1}{w(Q)}\int_Q \int_Q \frac{|f(x)-f(y)|^p}{|x-y|^{n+\delta p }}dy\, w(x)dx \right)^\frac{1}{p},
		\end{align*}
        where the weighted fractional Sobolev exponent $p^{*}_{\delta, r}$ is defined by the relation
	   \begin{equation}\label{eq:fractional-exponent}
		  \frac{1}{p}- \frac{1}{p^{*}_{\delta, r}}=\frac{\delta}{n} \frac{1}{r}.
	   \end{equation}
		
		\item Assume that $p \ge \frac{nr}{\delta}$. Then there exists $c_{n,p}>0$ such that, for every cube $Q$,
		\begin{align*}
			\left\| f-f_Q\right\|_{\exp L \left(Q, \frac{w(x)dx}{w(Q)} \right)} 
			\le & \, c_{n,p} [w]_{A_\infty} [w]_{A_p}^{\frac{1}{p}}[w]_{A_r}^{\frac{1}{p}} (1-\delta)^\frac{1}{p} \cdot \nonumber \\
			& \cdot \ell(Q)^\delta \left( \frac{1}{w(Q)}\int_Q \int_Q \frac{|f(x)-f(y)|^p}{|x-y|^{n+\delta p }}dy\, w(x)dx \right)^\frac{1}{p}.
		\end{align*}
	\end{itemize}
\end{corollary}

\begin{remark}
In the subcritical range $1\le p<\frac{nr}{\delta}$, the exponent $p_{\delta, r}^*$ is the natural fractional analogue of the Sobolev exponent. In particular, we show that this exponent is sharp in a certain range of $p$ and $r$, see Proposition \ref{prop: counterexample fractional p=1} and Remark {\ref{Fractional PS sharpness}} below.
\end{remark}

We conclude this section with the analogue of Corollary \ref{thm:fractional-PS} for pairs of weights. 

\begin{corollary}\label{thm:fractional-PS-pairs}
	Let $0<\delta<1$, $(u,v)\in A_p$, and $u\in A_r$ with $1\le r\le p$.

	\begin{itemize}
		\item Assume that $1\le p < \frac{nr}{\delta}$. Then there exists $c_{n}>0$ such that, for every cube $Q$,
		\begin{align*}
			\left\| f-f_{Q}\right\|_{\strt{2.3ex} L^{p^*_{\delta, r}, p} \left(Q, \frac{u(x)dx}{u(Q)} \right)}
			\le & \, c_n p^*_{\delta, r} [u]_{A_\infty} [u]_{A_r}^{\frac{\delta}{rn}} [u,v]_{A_p}^{\frac{1}{p}}(1-\delta)^\frac{1}{p}  \\
			& \cdot \ell(Q)^\delta \left( \frac{1}{u(Q)}\int_Q \int_Q \frac{|f(x)-f(y)|^p}{|x-y|^{n+\delta p }}dy\, v(x)dx \right)^\frac{1}{p},
		\end{align*}
        where the weighted fractional Sobolev exponent $p^{*}_{\delta, r}$ is defined in \eqref{eq:fractional-exponent}.
		
		\item Assume that $p \ge \frac{nr}{\delta}$. Then there exists $c_{n,p}>0$ such that, for every cube $Q$,
		\begin{align}\label{exp result fractional}
			\left\| f-f_Q\right\|_{\exp L \left(Q, \frac{u(x)dx}{u(Q)} \right)}
			\le & \, c_{n,p} [u]_{A_\infty} [u]_{A_r}^{\frac{1}{p}} [u,v]_{A_p}^{\frac{1}{p}} (1-\delta)^\frac{1}{p} \nonumber  \\
			& \cdot \ell(Q)^\delta \left( \frac{1}{u(Q)}\int_Q \int_Q \frac{|f(x)-f(y)|^p}{|x-y|^{n+\delta p }}dy\, v(x)dx \right)^\frac{1}{p}.
		\end{align}
	\end{itemize}
\end{corollary}

\begin{remark}
We remark that the Poincaré--Sobolev inequalities obtained in the preceding two subsections follow from a self-improving argument formulated in a general functional setting. This generality comes at the expense of a less precise dependence on the weight characteristic constants. For the specific functionals associated with the classical and fractional Poincaré inequalities, sharper quantitative estimates can be obtained; see \cite{ClarosJFA, LoristWagenaar}. In particular, when $p>1$, the corresponding estimates in \cite[Corollaries 4.4 and 5.4]{LoristWagenaar} show that, after replacing the Lorentz norm on the left-hand side with the corresponding Lebesgue norm, the power of $[w]_{A_\infty}$ can be improved from $1$ to $\tfrac{1}{p'}$.
\end{remark}

\subsection{\texorpdfstring{Weighted $L^{p,\infty}$--$L^{p^*,\infty}$ Poincaré--Sobolev inequality and applications to $A_1$ weights}{Weighted Lp,infinity--Lp*,infinity Poincare-Sobolev inequality and applications to A1 weights}}\label{subsection weak}

In this subsection, we investigate the weighted Poincaré-Sobolev inequality of type
\begin{equation*}
    \|f-f_Q\|_{Y_Q} \lesssim \ell(Q) \|\nabla f\|_{X_Q}, \quad \forall Q \subset \R^n,
\end{equation*}
where $X_Q$ and $Y_Q$ are Lorentz quasi-norms in $Q$ induced by the measure $\frac{w(x)\,dx}{w(Q)}$. We call it a weighted $X-Y$ Poincaré-Sobolev inequality. In Subsection {\ref{Section PS}}, we first prove a weighted $L^p-L^{p_r^*,\infty}$ Poincaré-Sobolev inequality \eqref{self improving example 1} and improve this result to a weighted $L^p-L^{p_r^*,p}$ Poincaré-Sobolev inequality \eqref{eq:PS-Lorentz-strong} by a truncation argument. In the present subsection, we improve \eqref{self improving example 1} from another perspective, replacing the right-hand side term with a smaller Lorentz norm. One result along this direction has already been shown in the previous subsections, namely \eqref{eq:PS-Lorentz-weak}. Here we show that, if $w\in A_1$, the right-hand side term can be replaced by the weak norm $L^{p,\infty}$. More precisely, we prove a weighted $L^{p,\infty}-L^{p^*,\infty}$ Poincaré-Sobolev inequality for every $w\in A_1$.

We first provide a Fefferman-Stein-type Poincaré-Sobolev inequality via an alternative approach, independent of the self-improving theory developed above. The estimate has weak norms on both sides and the maximal function of the weight on the right-hand side. We do not impose any assumptions on the weight. 

\begin{theorem}\label{Thm PS 4}
    Let $w$ be a weight in $\R^n$ with $n\ge 2$ and let $1<p<n$. Suppose $f$ is a Lipschitz function. Then, there exists a constant $c_{n,p}>0$ such that
        \begin{equation}\label{Thm PS 4 Eq}
            \left\| f-f_Q\right\|_{ \strt{2.3ex} L^{p^*,\infty}(Q,w)} \le c_{n,p} \left\| \left| \nabla f \right| \frac{M^c(w\chi_Q)^\frac{1}{n'}}{w} \right\|_{\strt{2.3ex} L^{p,\infty}(Q,w)}
        \end{equation}
    for each cube $Q$, where $\frac{1}{p^*}=\frac{1}{p}-\frac{1}{n}$ and $M^c$ denotes the centered Hardy-Littlewood maximal function.
\end{theorem}

The previous result is in the spirit of \cite[Theorem 1.21]{PR}, which it improves for the range $p > 1$ in two significant ways. First, the right-hand side involves the weaker Lorentz quasi-norm $L^{p,\infty}$. Second, the oscillation on the left-hand side is measured with respect to the unweighted average $f_Q$, rather than the weighted version $f_{Q,w}$. Furthermore, \cite[Theorem 1.21]{PR} can be recovered as a corollary by combining \eqref{Thm PS 4 Eq} with a standard truncation argument. We provide further details in Section \ref{Section Weak}.

For $A_1$ weights, we can derive the following weighted Poincaré-Sobolev inequality with weak norms.

\begin{corollary}\label{Thm PS 5}
    Let $w\in A_1$ be a weight in $\R^n$ with $n\ge 2$ and let $1<p<n$. Then there exists a constant $c_{n,p}>0$ such that
        \begin{equation}\label{Thm PS 5 Eq}
            \left\| f-f_Q\right\|_{\strt{2.3ex} L^{p^*,\infty}\left(Q,\frac{w(x)dx}{w(Q)}\right)}
            \le c_{n,p} [w]_{A_1} \ell(Q)
            \left\| \nabla f \right\|_{\strt{2.3ex} L^{p,\infty}\left(Q,\frac{w(x)dx}{w(Q)}\right)}
        \end{equation}
    for each cube $Q$.
\end{corollary}

The unweighted global version of \eqref{Thm PS 5 Eq} is known for the Lebesgue measure \cite{Peetre}; see also \cite{Maly} for a different approach.

\begin{remark}
Regarding the optimal power of the $A_1$ constant in \eqref{Thm PS 5 Eq}, we will show that $\frac{1}{p}$ is a lower bound for \eqref{Thm PS 5 Eq} to hold. Indeed, if \eqref{Thm PS 5 Eq} were to hold with a smaller exponent $[w]_{A_1}^{\alpha}$, $\alpha<1/p$, then, since $\|\nabla f\|_{L^{p,\infty}\left(Q,w\right)} \le \|\nabla f\|_{L^{p}\left(Q,w\right)}$, the same bound would follow with the strong $L^p$ norm of the gradient on the right-hand side. Combining this with the truncation method would then yield an improvement of the sharp weighted Poincaré-Sobolev inequality obtained in \cite[Theorem 2.4]{ClarosJFA}, and hence lead to a contradiction.
\end{remark}

\subsection*{Outline of the paper}
The paper is organized as follows. In Section \ref{Section 2}, we collect definitions and known results on Lorentz spaces, Muckenhoupt weights, and the truncation method. In Section \ref{good lambda}, we prove our main self-improving result under the assumption $w\in A_r$, including the exponential endpoint, and we discuss sharpness. Section \ref{Section 4} contains the proofs of Theorems \ref{Selfimproving Ainfty} and \ref{Thm1}. In Sections \ref{Section 5} and \ref{Section 6}, we study the applications from Subsections \ref{Section PS} and \ref{Section FPS}, respectively. In Section \ref{Section Weak}, we provide the proof of the result stated in Subsection \ref{subsection weak}. Finally, the appendices contain extensions to polynomial-type oscillations, vector-valued inequalities, and the rectangular setting.

\section{Some preliminaries and known results}\label{Section 2}

\subsection{Lorentz spaces}

We briefly recall the definition and some basic properties of Lorentz spaces. Let $\mu$ be a Borel measure on $\R^n$ and let $E\subset\R^n$ be a $\mu$-measurable set. For $0<p<\infty$ and $0<q\le\infty$, the Lorentz space $L^{p,q}(E,\mu)$ consists of all measurable functions $f$ such that
\begin{equation}\label{Lpq quasinorm}
	\|f\|_{\strt{2.3ex} L^{p,q}(E,\mu)}
	:=\begin{cases}
		\left(p\displaystyle\int_0^\infty
		t^{q}\,\mu\big(\{x\in E:\ |f(x)|>t\}\big)^{\frac{q}{p}}\,\frac{dt}{t}\right)^{\frac{1}{q}},
		& 0<q<\infty,\\[1.1em]
		\displaystyle\sup_{t>0}\, t\,\mu\big(\{x\in E:\ |f(x)|>t\}\big)^{1/p}, & q=\infty,
	\end{cases}
\end{equation}
is finite. The functional $\|\,.\,\|_{L^{p,q}(E,\mu)}$ is a norm if and only if either $1\le q\le p$ or $p=q=\infty$; otherwise it is a quasi-norm. Moreover, if $1<p<\infty$ and $1\le q\le\infty$ (as well as in the endpoint cases $p=q=1$ and $p=q=\infty$), the quasi-norm \eqref{Lpq quasinorm} is equivalent to a norm on $L^{p,q}(E,\mu)$; see \cite{BS}.

We will use the following form of H\"older's inequality in Lorentz spaces (see \cite{KS10}): if $1\le p<\infty$, $1\le q\le\infty$, and $f\in L^{p,q}(E,\mu)$, $g\in L^{p',q'}(E,\mu)$, then
\begin{equation}\label{Holder Lorentz}
	\int_{E} |f(x)g(x)|\,d\mu(x)
	\le \|f\|_{\strt{2.3ex} L^{p,q}(E,\mu)}\,\|g\|_{ \strt{2.3ex} L^{p',q'}(E,\mu)},
\end{equation}
where $p'=\frac{p}{p-1}$ is the usual conjugate exponent (with the convention $p'=1$ if $p=\infty$ and $p'=\infty$ if $p=1$).

We also recall the nestedness of the scale with respect to the second index (see \cite{BS}): if $0<p\le\infty$ and $0<q_1\le q_2\le\infty$, then
\begin{equation}\label{eq:lorentz-nestedness}
	\|f\|_{\strt{2.3ex} L^{p,q_2}(E,\mu)} \le \left(\frac{q_1}{p}\right)^{\frac{1}{q_1}-\frac{1}{q_2}}
\|f\|_{\strt{2.3ex} L^{p,q_1}(E,\mu)}.
\end{equation}

Finally, we will use the standard embedding between Lorentz spaces on a probability space (see \cite{Hunt}). Let $(X,\nu)$ be a measure space with $\nu(X)=1$. If $0<p_1<p_2<\infty$ and $0<q<\infty$,
then
\begin{equation}\label{Lorentz inclusion}
	\|f\|_{\strt{2.3ex} L^{p_1,q}(X,\nu)}
	\le \Big(\frac{p_1p_2}{q(p_2-p_1)}\Big)^{\frac{1}{q}}\,\|f\|_{\strt{2.3ex} L^{p_2,\infty}(X,\nu)}.
\end{equation}

\subsection{Maximal operators and \texorpdfstring{$A_p$}{Ap} theory of weights}

In this section, we collect notation and basic properties of Muckenhoupt weights and maximal operators. A weight is a nonnegative, locally integrable function, usually denoted by $w$.

Given a locally integrable function $f\in L^1_{loc}(\R^n)$, the (uncentered) Hardy--Littlewood maximal operator is defined by
	\begin{equation*}
			M f(x) = \sup_{x\in Q}  \frac{1}{|Q|}\int_Q |f(y)|dy,
	\end{equation*}
		where the supremum is taken over all cubes $Q\subset \R^n$ with sides parallel to the coordinate axes and such that $x\in Q$. The centered version of the maximal operator is defined by
\begin{equation*}
	M^c f(x) = \sup_{r>0} \frac{1}{|Q(x,r)|} \int_{Q(x,r)}|f(y)|dy,
\end{equation*}
where $Q(x,r)$ denotes the cube centered at $x$ with side length $r$. Since Lebesgue measure is doubling, $M$ and $M^c$ are pointwise comparable; in particular, any estimate for one operator implies the corresponding estimate for the other with dimensional constants.

Let $1<p<\infty$. A weight $w$ belongs to $A_p$ if
\begin{equation*}
[w]_{A_p} :=\sup_Q \left(\frac{1}{|Q|}\int_Q w(x)\,dx\right)
\left(\frac{1}{|Q|}\int_Q w(x)^{1-p'}\,dx\right)^{p-1} <\infty,
\end{equation*}
where the supremum is over all cubes $Q$. For $p=1$ we write $w\in A_1$ if
\begin{equation*}
[w]_{A_1} :=\sup_Q \left(\frac{1}{|Q|}\int_Q w(x)\,dx\right)\,
\operatorname*{ess\,sup}_{x\in Q} w(x)^{-1} <\infty.
\end{equation*}
Equivalently,
\begin{equation*}
Mw(x)\le [w]_{A_1}\,w(x)\qquad\text{for a.e. }x\in\R^n.
\end{equation*}

The boundedness of $M$ on weighted spaces is characterized by these classes (see \cite{M}):
\begin{equation*}
	M:L^p(w)\longrightarrow L^p(w)\ \text{ if and only if }\ w\in A_p,\qquad 1<p<\infty,
\end{equation*}
and, 
\begin{equation*}
	M:L^p(w)\longrightarrow L^{p,\infty}(w)\ \text{ if and only if }\ w\in A_p, \qquad 1\le p<\infty.
\end{equation*}

We will frequently use the following property of $A_p$ weights: a weight $w$ belongs to the class $A_p$ if and only if  for every measurable function $f\ge 0$, we have
\begin{equation}\label{Ap 1}
	\frac{1}{|Q|}\int_Q f(x)dx \le [w]^\frac{1}{p}_{A_p} \left( \frac{1}{w(Q)}\int_Q f(x)^p w(x) dx\right)^\frac{1}{p}
\end{equation}
for every cube $Q$. In particular, taking $f=\chi_E$ for a measurable set $E\subset \R^n$, we obtain 
\begin{equation}\label{Ap 2}
	\frac{|E|}{|Q|}\le [w]^\frac{1}{p}_{A_p} \left( \frac{w(E)}{w(Q)}\right)^\frac{1}{p}.
\end{equation}

We will also consider the class of pairs of $A_p$ weights.

\begin{definition}\label{def:pair}
	Let $(u,v)$ be a pair of weights.
	\begin{enumerate}
		\item For $1<p<\infty$ we say that $(u,v)\in A_p$ if
			\begin{equation*}
				[u,v]_{A_p}= \sup_Q \left( \frac{1}{|Q|}\int_Q u(x)dx\right)\left( \frac{1}{|Q|}\int_Q v(x)^{1-p'} dx\right) ^{p-1} <\infty.
			\end{equation*}
		\item We say that $(u,v)\in A_1$ if
			\begin{equation*}
				[u,v]_{A_1}= \sup_Q \left( \frac{1}{|Q|}\int_Q u(x)dx\right) \underset{ Q}{\operatorname{ess} \sup} \left(v^{-1}\right)	<\infty,
			\end{equation*}
			or equivalently, if $Mu(x)\le [u,v]_{A_1} v(x)$ a.e. $x\in \R^n$.
	\end{enumerate}
\end{definition}

We write
\begin{equation*}
	A_\infty = \bigcup_{p\ge 1} A_p.
\end{equation*}
We will quantify membership of a weight in $A_\infty$ using the Fujii--Wilson constant
\begin{equation*}
	[w]_{A_\infty} = \sup_Q \frac{1}{w(Q)} \int_Q M(w\chi_Q)(x)dx,
\end{equation*}
where the supremum is taken over all cubes $Q\subset\R^n$.

We will also use a weight class naturally adapted to Lorentz norms, introduced by Chung, Hunt, and Kurtz in~\cite{CHK}.

\begin{definition}\label{def:CHK}
	Let $1<p<\infty$ and let $1\le q <\infty$. We denote by $A_{p,q}$ the class of weights $w$ such that
	\begin{equation*}
		[w]_{A_{p,q}}= \sup_Q \left( \frac{1}{|Q|}\int_Q w(x)dx \right) \left\| \frac{1}{w}\right\|_{\strt{2.3ex} L^{p',q'} \left( Q, \frac{w(x)dx}{|Q|}\right)}^p<\infty.
	\end{equation*}
\end{definition}
Using \eqref{Holder Lorentz}, one readily checks that $w\in A_{p,q}$ implies the following weighted estimate analogous to \eqref{Ap 1}: for every measurable $f\ge 0$ we have,
\begin{equation}\label{Apq}
	\frac{1}{|Q|}\int_Q f(x)dx \le [w]^\frac{1}{p}_{A_{p,q}} \left\| f \right\|_{L^{p,q} \left( Q, \frac{w(x)dx}{w(Q)}\right)}
\end{equation}
for every cube $Q$. In particular, for $1<p<\infty$ and $1<q<\infty$ one has $A_{p,q}=A_p$. At the endpoint $q=1$, the class $A_{p,1}$ strictly contains $A_p$ and is sometimes referred to as the restricted $A_p$ class, and denoted by $A_p^\mathcal R$.

\subsection{Truncation method}

The truncation method is a classical technique for upgrading weak-type Poincaré-Sobolev inequalities to stronger results along the Lorentz scale. In the first-order setting, this idea already appears in the work of Maz'ya \cite[p. 110]{Mazya}; see also \cite{Hajlasz, KO03} for more details. We will use a non-standard version of this argument, based on a trick from \cite{FPW98}, which allows us to keep the unweighted average $f_Q$ on the left-hand side. In order to do this, besides the weak-type Poincaré-Sobolev inequality, we require an additional unweighted $L^1$ Poincaré inequality.

\begin{theorem}\label{thm:truncation-method}
    Let $1\le p\le q<\infty$, $1\le r\le p$, let $Q\subset\R^n$ be a cube, and let $\mu$ and $\nu$ be positive Borel measures on $\R^n$ such that $\mu(Q)=1$. Assume that there exists a constant $C>0$ such that for every $f\in Lip(Q)$,
	\begin{equation*}
		\left\| f-f_Q\right\|_{\strt{2.3ex} L^{q,\infty}(Q,\mu)}\le C\left\|\nabla f \right\|_{\strt{2.3ex} L^{p,r}(Q,\nu) }
	\end{equation*}
	and
	\begin{equation*}
		\frac{1}{|Q|}\int_Q |f(x)-f_Q|\,dx\le C\left\|\nabla f \right\|_{\strt{2.3ex} L^{p,r}(Q,\nu) }.
	\end{equation*}
	Then, for every $f\in Lip(Q)$, we have
	\begin{equation*}
		\left\| f-f_Q\right\|_{\strt{2.3ex} L^{q,p}(Q,\mu)}\le 10\, C\left\|\nabla f \right\|_{\strt{2.3ex} L^{p,r}(Q,\nu) }.
	\end{equation*}
\end{theorem}

The previous theorem can be proved by following the strategy in  \cite[Theorem 3.1]{FPW98} and using the summability property of Lorentz quasi-norms stated in Lemma \ref{Lemma Lorentz} below; this is also the point where the restriction $1\le r\le p$ enters.

In \cite{DLV21}, it was proved that the truncation method is applicable to the fractional seminorm. We will use the following version of the truncation method; this version combines the result of \cite{DLV21} and the weak-implies-strong argument of \cite{KO03} to improve it up to the Lorentz norm $L^{q,p}$. We also use the ideas of \cite[Theorem~3.1]{FPW98} in order to keep the unweighted average $f_Q$ on the left-hand side. We refer to \cite[Proposition A.2]{LoristWagenaar} for an explicit proof of a related result.

\begin{theorem}\label{Truncation method}
    Let $1\le p\le q<\infty$, let $Q\subset\R^n$ be a cube, let $\mu$ be a positive Borel measure on $\R^n$ such that $\mu(Q)=1$, and let $K:\R^n\times\R^n\to[0,\infty)$ be a positive measurable function. Assume that there exists $C>0$ such that for every $f\in Lip(Q)$,
	\begin{equation*}
		\left\| f-f_Q\right\|_{\strt{2.3ex} L^{q,\infty}(Q,\mu)} \le C\left(\int_Q\int_Q |f(y)-f(z)|^p K(y,z)\,dy\,dz\right)^\frac{1}{p},
	\end{equation*}
	and
	\begin{equation*}
		\frac{1}{|Q|}\int_Q |f(x)-f_Q|\,dx \le C\left(\int_Q\int_Q |f(y)-f(z)|^p K(y,z)\,dy\,dz\right)^\frac{1}{p}.
	\end{equation*}
	Then, for every $f\in Lip(Q)$, we have
	\begin{equation*}
		\left\| f-f_Q\right\|_{\strt{2.3ex} L^{q,p}(Q,\mu)} \le 130 \, C\left(\int_Q\int_Q |f(y)-f(z)|^p K(y,z)\,dy\,dz\right)^\frac{1}{p}.
	\end{equation*}
\end{theorem}

\subsection*{Notation}
As usual, $c$ denotes a positive constant, possibly varying from line to line. We write $c_{\alpha, \beta, ...}$ to denote a constant depending only on $\alpha, \beta,...\,$.

\section{Proof of Theorem \ref{Selfimproving Ar} and Sharpness} \label{good lambda}

We begin this section by establishing a refinement of \cite[Proposition 2.3]{CP}, which will be used in the proof of Theorem \ref{Selfimproving Ar} to prove the exponential-type estimate.

\begin{proposition}\label{prop:linearweakLptoexpL}
Let $(X,\mu)$ be a probability space and let $f\ge 0$ be measurable. Assume that there exist constants $\gamma>0$ and $p_0\ge 1$ such that for every $p>p_0$,
\begin{equation}\label{prop:eqweakLp}
\|f\|_{\strt{2.3ex} L^{p,\infty}(X,\mu)}\le \gamma\,p.
\end{equation}
Then $f\in \exp L (X,\mu)$, in the sense that 
\begin{equation*}
\mu(\{x\in X:\ f(x)>t\})\le e^{p_0} \exp\Big(-\frac{t}{e\,\gamma}\Big),
\end{equation*}
for every $t>0$. Moreover, there exists an absolute constant $C>0$ such that
\begin{equation*}
\|f\|_{\exp L(X,\mu)} \le C \, p_0 \, \gamma  .
\end{equation*}
\end{proposition}

\begin{proof}
Fix $t>e\gamma p_0$. Choose $p=p(t):=\frac{t}{e\gamma}>p_0$, so that $\gamma p/t=e^{-1}$. Applying \eqref{prop:eqweakLp} with this choice of $p$ gives
\begin{equation*}
    \mu(\{x\in X: f(x)>t\}) \le \Big(\frac{1}{e}\Big)^{p(t)} = \exp\Big(-\frac{t}{e\gamma}\Big),
\end{equation*}
for every $t>e\gamma p_0$.  If $0<t\le e\gamma p_0$, then $\mu(\{x\in X: f(x)>t\})\le 1$ and hence
\begin{equation*}
    \mu(\{x\in X: f(x)>t\})
\le e^{p_0}\exp\Big(-\frac{t}{e\gamma}\Big),
\end{equation*}
since $\exp(-t/(e\gamma))\ge e^{-p_0}$ in the range $0<t\le e\gamma p_0$. Therefore, for every $t>0$,
\begin{equation*}
    \mu(\{x\in X: f(x)>t\})\le e^{p_0}\exp\Big(-\frac{t}{e\gamma}\Big).
\end{equation*}
\end{proof}

With Proposition \ref{prop:linearweakLptoexpL} in hand, we now turn to the proof of Theorem \ref{Selfimproving Ar}.

\begin{proof}[Proof of Theorem \ref{Selfimproving Ar}]
We first show that estimate (\ref{Thm Ar eq p>s}) follows directly from estimate (\ref{Thm Ar eq}) and Proposition \ref{prop:linearweakLptoexpL}. Suppose that $a\in SD_p^s(w)$ with $p \ge rs$. Then, for every $s_0 \ge s$,  
\begin{align}
 \sum\limits_{j}a(Q_j)^pw(Q_j) & \le \|a\|_{SD_{p}^{s}(w)}^p
\left(\frac{\left| \bigcup_j Q_j \right| }{|Q|}  \right)^{\frac{p}{s}}a(Q)^{p}w(Q) \nonumber \\
& \le \|a\|_{SD_{p}^{s}(w)}^p\left( \frac{\left| \bigcup_j Q_j \right| }{|Q|}\right)^{\frac{p}{s_0}}a(Q)^{p}w(Q), \nonumber 
\end{align}
for any family $\{Q_j\}_j$ of pairwise disjoint subcubes of the cube $Q$. Hence $\|a\|_{SD_{p}^{s_0}(w)} \le \|a\|_{SD_{p}^{s}(w)}$. For each $p < q$, we choose $s_0$ to be the constant satisfying the relation 
$$\frac{1}{p}-\frac{1}{q}=\frac{1}{r{s}_0}.$$ 
Since $p < rs_0$, applying estimate (\ref{Thm Ar eq}) gives 
\begin{equation*}
    \norm{f-f_Q}_{\strt{2.3ex} L^{q,\infty}\left(Q, \frac{w(x)dx}{w(Q)}\right)} \le c_n q  [w]_{A_\infty} [w]_{A_r}^{\frac{1}{p}}\left\| a\right\|_{SD_p^s(w)} a(Q),
\end{equation*}
for every $q > p \ge 1$.  We can apply Proposition \ref{prop:linearweakLptoexpL} with 
\begin{equation*}
    \gamma = c_{n}   [w]_{A_\infty} [w]_{A_r}^{\frac{1}{p}}\left\| a\right\|_{SD_p^s(w)} a(Q), 
\end{equation*}
and we obtain the desired result (\ref{Thm Ar eq p>s}).

We now prove the estimate (\ref{Thm Ar eq}). Fix a cube $Q$, our goal is to show that for every $t>0$,
\begin{equation}\label{displayThm Ar eq}
t^{p_{r,s}^{*}} \, w(\{x\in Q: |f(x)-f_Q|>t\})\leq  \|a\|^{p_{r,s}^{*}} (c_n \, p_{r,s}^{*}\, [w]_{A_\infty} [w]_{A_r}^{\frac{1}{rs}} a(Q))^{p_{r,s}^{*}}\,  \, w(Q).
\end{equation}
Denote by $M_Q$ the dyadic maximal operator localized in $Q$. Since 
\begin{equation*}
	|f(x)-f_Q| \leq M_Q (f-f_Q)(x),
\end{equation*}
we estimate the larger set 
\begin{equation*}
	\Omega_t = \{x\in Q: M_Q(f-f_Q)(x)>t\}.
\end{equation*} 
Let $\{Q_j\}_j$ be the maximal dyadic cubes forming $\Omega_t$. Then, either $\{Q_j\}_j =\{Q\}$ or the cubes $Q_j$ satisfy the condition
\begin{equation}{\label{Calderon-Zygmund decomposition}}
t < \frac{1}{|Q_j|}\int_{Q_j} {|f-f_Q|} \leq 2^n t    
\end{equation}
and for almost every $x \in Q \setminus \cup_j Q_j$, $|f(x)-f_Q| \le t$. Let $\kappa=2^n+1$. We claim that 
\begin{equation*}
	w(\Omega_{\kappa t}) \leq \sum_j w(E_{Q_j}),
\end{equation*}
where
\begin{equation*}
	E_{Q_j}= \{ x\in Q_j: M_{Q_j}(f-f_{Q_j})(x)>t\}.
\end{equation*} 
When $\{Q_j\}_j =\{Q\}$, this claim holds trivially. In the other case, since $\kappa > 1$, it is not difficult to show that, except for a set of null measure,
\begin{equation*}
	\Omega_{\kappa t} \subset \bigcup_{j}\{x \in Q_j : M_{Q_j}(f-f_Q)(x) > \kappa t\}.
\end{equation*}
Moreover, for each $j$, by (\ref{Calderon-Zygmund decomposition}), we have $|f_{Q_j} - f_Q| \le 2^n t$. Applying these two observations, we have, for every $x \in Q_j$ satisfying $M_{Q_j}(f-f_Q)(x) > \kappa t$, 
\begin{equation*}
	M_{Q_j}(f-f_{Q_j}) \ge |M_{Q_j}(f-f_{Q})-|f_Q-f_{Q_j}|| > (\kappa-2^n)t =  t.
\end{equation*}
This proves the claim.

Note that for every $\gamma >0$,
\begin{align*}
E_{Q_j}  \subseteq & \{x\in Q_j : M_{Q_j}(f-f_{Q_j})(x)>t, \, M_{Q}^\# f (x) \leq \gamma t\}  \cup \{ x\in Q_j :   M_{Q}^\# f (x) > \gamma t\} \\
 :=&  A_j \cup B_j,
\end{align*}
where 
\begin{equation*}
	M_{Q}^\# f (x) = \sup_{x\in R\in \mathcal{D}(Q)} \frac{1}{|R|}\int_R |f(y)-f_R|dy.
\end{equation*}
Thus, we have 
\begin{equation*}
	w(\Omega_{\kappa t}) \leq \sum_j w(E_{Q_j}) \leq \sum_j w(A_j) + \sum_j w(B_j).
\end{equation*}

For the sets $A_j$, we use the good-lambda type inequality  \cite[Corollary 1.4]{CP}.
\begin{equation*}
	\sum_j w(A_j) \leq c_1 e^{-\frac{c_2}{ [w]_{A_{\infty}} \gamma} } \sum_j w(Q_j) = c_1 e^{-\frac{c_2}{ [w]_{A_{\infty}} \gamma} } w(\Omega_t),
\end{equation*}
where $c_1,c_2>0$ are dimensional constants. On the other hand, for $B_j$ we can argue as follows. We have
\begin{align*}
\bigcup_j B_j \subseteq \{x\in Q: M_{Q}^\# f(x)>\gamma t\}:=\bigcup_i R_i,
\end{align*}
where $R_i$ are the maximal dyadic subcubes of $Q$ such that 
\begin{equation*}
	\gamma t < \frac{1}{|R_i|}\int_{R_i}|f(x)-f_{R_i}|dx .
\end{equation*}
By hypothesis \eqref{starting-point0}, we know that $\gamma t < a(R_i)$. We now use the condition that $a$ satisfies, namely $a\in SD_{p}^{s}(w)$. Then,
\begin{align}
	\sum_j w(B_j)  \leq & w\big( \{x\in Q: M_{Q}^\# f(x)>\gamma t\} \big) \nonumber\\
    = & w\left(\bigcup_i R_i\right) {\label{first step}}\\
	 \le & \left( \frac{1}{\gamma t}\right)^p\sum_i a(R_i)^p w(R_i) \nonumber \\
	 \le & \|a\|^p \left( \frac{1}{\gamma t}\right)^p \, \left(\frac{|\bigcup_iR_i|}{|Q|} \right )^{ \frac{p}{s} } a(Q)^p w(Q). {\label{middle step}} \\
\le & \|a\|^{p}[w]_{A_r}^{\frac{p}{rs}}\left( \frac{1}{\gamma t}\right)^p \, \left(\frac{w\left(\bigcup_iR_i\right)}{w(Q)} \right )^{\frac{p}{r s} } a(Q)^p w(Q). \nonumber
\end{align}
Hence, combining this with \eqref{first step} and recalling that
$$
\frac{1}{p_{r,s}^{*}}= \frac{1}{p}- \frac{1}{sr}, 
$$
we have,
\begin{equation*}
	\left(\frac{w(\bigcup_iR_i)}{w(Q)} \right)^{1-\frac{p}{rs}} \le [w]_{A_r}^{\frac{p}{rs}} \left( \frac{1}{\gamma t}\right)^{p}\|a\|^{p}a(Q)^{p},
\end{equation*}
and hence
\begin{equation*}
	\frac{w(\bigcup_iR_i)}{w(Q)} \le [w]_{A_r}^{\frac{p_{r,s}^{*}}{rs}} \left( \frac{1}{\gamma t}\right)^{p_{r,s}^{*}}\|a\|^{p_{r,s}^{*}}a(Q)^{p_{r,s}^{*}}.
\end{equation*}
Combining these results, it follows that,  
\begin{align*}
  w(\Omega_{\kappa t}) & \leq c_1 e^{-\frac{c_2}{ [w]_{A_{\infty}}\gamma } }w(\Omega_t) + 
[w]_{A_r}^{\frac{p_{r,s}^{*}}{rs}} \left( \frac{1}{\gamma t }\right)^{p_{r,s}^{*}} \|a\|^{p_{r,s}^{*}}a(Q)^{p_{r,s}^{*}} w(Q)
\end{align*}
and then 
\begin{align*}
(\kappa t)^{p_{r,s}^{*}} w(\Omega_{\kappa t}) & \leq c_1 e^{-\frac{c_2}{ [w]_{A_{\infty}}\gamma } }\, (\kappa t)^{p_{r,s}^{*}}w(\Omega_t) + 
[w]_{A_r}^{\frac{p_{r,s}^{*}}{rs}} \left( \frac{\kappa}{\gamma  }\right)^{p_{r,s}^{*}} \|a\|^{p_{r,s}^{*}}a(Q)^{p_{r,s}^{*}} w(Q).
\end{align*}

We define
\begin{equation*}
	\varphi(N) = \sup_{0<t\leq N} t^{p_{r,s}^{*}} w(\Omega_t),
\end{equation*}
since this function is increasing, we have
\begin{equation*}
	\varphi(N) \leq \varphi(\kappa N) \leq c_1 \kappa^{p_{r,s}^{*}} e^{-\frac{c_2}{ [w]_{A_{\infty}} \gamma} }\varphi(N) + 
[w]_{A_r}^{\frac{p_{r,s}^{*}}{rs}} \left( \frac{\kappa}{\gamma}\right)^{p_{r,s}^{*}} \|a\|^{p_{r,s}^{*}}a(Q)^{p_{r,s}^{*}} w(Q).
\end{equation*}
Since the parameter $\gamma$ remains free throughout the derivation, we may set it such that  
\begin{equation*}
		c_1 \kappa^{p_{r,s}^{*}} e^{-\frac{c_2}{[w]_{A_{\infty}}\gamma}}=\frac 12,\end{equation*}
then 
\begin{equation*}
	\frac1{\gamma} \le   c_n[w]_{A_\infty} p_{r,s}^{*}.
\end{equation*}
This yields the estimate
\begin{equation*}
	\varphi(N) \le c_n^{p_{r,s}^{*}}[w]_{A_\infty}^{p_{r,s}^{*}}[w]_{A_r}^{\frac{p_{r,s}^{*}}{rs}}  ({p_{r,s}^{*}})^{p_{r,s}^{*}} \|a\|^{p_{r,s}^{*}}a(Q)^{p_{r,s}^{*}} w(Q).
\end{equation*}
Passing to the limit $N \to \infty$ in the above inequality yields \eqref{displayThm Ar eq}, concluding the proof of the theorem.
\end{proof}

Before addressing the sharpness of Theorem \ref{Selfimproving Ar}, we present a consequence of Proposition \ref{prop:linearweakLptoexpL}. Specifically, the following result shows that the condition $w\in A_\infty$ is necessary for the $D_p$-type self-improving property (Theorem \ref{Thm 1.6 CP}) to hold. More precisely, if the conclusion of Theorem \ref{Thm 1.6 CP} holds for every functional $a\in D_p(w)$, then $w$ must belong to $A_\infty$, with quantitative control of $[w]_{A_\infty}$.

\begin{proposition}\label{prop:Ainfty-necessity-CP}
	The assumption $w\in A_\infty$ in Theorem \ref{Thm 1.6 CP} is necessary. More precisely, let $w$ be a weight and assume that there exists a constant $C_w<\infty$ such that, for every $1\le p<\infty$, every functional $a\in D_p(w)$, and every locally integrable function $f$ satisfying
	\begin{equation*}
		\frac{1}{|Q|}\int_Q |f(x)-f_Q|\,dx \le a(Q)
	\end{equation*}
	for every cube $Q$, we have
	\begin{equation}\label{eq:CPnecessity-Ainfty}
		\left\| f-f_Q\right\|_{ L^{p,\infty}\left(Q,\frac{w(x)\,dx}{w(Q)}\right)}
		\le C_w p\, \left\|a\right\|_{D_p(w)} a(Q)
	\end{equation}
	for every cube $Q$. Then $w\in A_\infty$ with $[w]_{A_\infty}\le c_n C_w$.
\end{proposition}

\begin{proof}
	Assume that \eqref{eq:CPnecessity-Ainfty} holds for a weight $w$. Let $f\in \operatorname{BMO}$ and define $a(Q):=\|f\|_{\operatorname{BMO}}$ for every cube $Q$. Then,
	\begin{equation*}
		\frac{1}{|Q|}\int_Q |f(x)-f_Q|\,dx \le a(Q)
	\end{equation*}
	for every cube $Q$. Moreover, $a\in D_p(w)$ for every $1\le p<\infty$ and $\|a\|_{D_p(w)}= 1$. Applying \eqref{eq:CPnecessity-Ainfty}, we obtain, for every cube $Q$ and every $1<p<\infty$,
	\begin{equation*}
		\left\| f-f_Q\right\|_{ L^{p,\infty}\left(Q,\frac{w(x)\,dx}{w(Q)}\right)}
		\le C_w p\, \|f\|_{\operatorname{BMO}}.
	\end{equation*}
	By Proposition \ref{prop:linearweakLptoexpL}, it follows that
	\begin{equation*}
		\left\| f-f_Q\right\|_{\exp L\left(Q,\frac{w(x)\,dx}{w(Q)}\right)}
		\le C\,C_w\,\|f\|_{\operatorname{BMO}}.
	\end{equation*}
	  Hence $f\in \operatorname{BMO}_{w,w}$ and
	\begin{equation}\label{eq:CPnecessity-2}
		\|f\|_{\operatorname{BMO}_{w,w}}\le C\,C_w\,\|f\|_{\operatorname{BMO}},
	\end{equation}
    where $\|f\|_{\operatorname{BMO}_{w,w}} : = \sup_Q \frac{1}{w(Q)}\int_Q |f-f_Q| w$. Thus we have proved $\operatorname{Id}: \operatorname{BMO} \hookrightarrow \operatorname{BMO}_{w,w}$. However, by the characterization proved in \cite[Theorem 1.2]{OPRR}, the embedding $\operatorname{Id}: \operatorname{BMO} \hookrightarrow \operatorname{BMO}_{w,w}$ holds if and only if $w\in A_\infty$, therefore $w\in A_\infty$. Moreover, taking the supremum in \eqref{eq:CPnecessity-2} over all $f\in \operatorname{BMO}$ with $\|f\|_{\operatorname{BMO}}= 1$, and using again \cite[Theorem 1.2]{OPRR}, we obtain $[w]_{A_\infty}\le c_n C_w$.
\end{proof}
\ed

Now we discuss the sharpness of \eqref{Thm Ar eq} in Theorem {\ref{Selfimproving Ar}}.  In this context, sharpness means that the weak norm $L^{p_{r,s}^{*},\infty}$ on the left-hand side cannot be replaced by an $L^{q,\infty}$ norm for any $q > p_{r,s}^{*}$.   We construct a counterexample showing that  \eqref{Thm Ar eq} is sharp in the case $s\ge n$ and $w\in A_r$ with $1\le r\le p$.

We start with the case $s=n$ and $1\le r\le p< n r$. By using \eqref{I1 P11} and (\ref{Ap 1}), we have that for $w \in A_p$,
\begin{equation}{\label{weight 1,1 Poincare inequality}}
\frac{1}{|Q|}\int_Q |f(x)-f_Q|dx \le c_n [w]_{A_p}^\frac{1}{p} \ell(Q)\left(\frac{1}{w(Q)}\int_{Q}|\nabla f(x)|^p w(x) dx\right)^{\frac{1}{p}}:=a(Q).
\end{equation}
Moreover, from \cite[Lemma 3.2]{PR}, we know that $a \in SD_p^{n}(w)$. Then, applying Theorem {\ref{Selfimproving Ar}}, we obtain
\begin{equation} \label{Thm Ar eq-bis}
	\norm{f-f_Q}_{\strt{2.3ex}L^{p_r^{*},\infty}\left(Q, \frac{w(x)\,dx}{w(Q)}\right)} \le c_n p^*_r [w]_{A_\infty}\,[w]_{A_p}^{\frac{1}{p}}\,[w]_{A_r}^{\frac{1}{rn}}\, \ell(Q) \left( \frac{1}{w(Q)}\int_{Q}|\nabla f(x)|^p\, w(x)\,dx\right)^{\frac{1}{p}},
\end{equation}
for every cube $Q$, where $\frac{1}{p_r^{*}} = \frac{1}{p} - \frac{1}{nr}$, providing the proof of \eqref{self improving example 1}. 
The application of the truncation argument from Theorem \ref{thm:truncation-method} yields the proof of Corollary \ref{thm:PS-Lorentz-strong}. The proof of Corollary \ref{prop:pairs} is analogous and follows the same lines.

In the next proposition, we demonstrate that the exponent $p_r^{*}$ on the left-hand side of \eqref{Thm Ar eq-bis} cannot be improved, thereby establishing the sharpness of Theorem \ref{Selfimproving Ar} in the case $s=n$ and $p<nr$. 

For the sake of simplicity, in the following two propositions, we use the notation $A \lesssim B$ to indicate that there exists a constant $C$, depending on $n$, $p$, $q$, $[w]_{A_r}$, and on any fixed auxiliary cut-off function chosen in the proof, such that $A \le CB$.

\begin{proposition}{\label{counterexample proposition}} 
Suppose $1\le r\le p$ and $0<\frac{1}{q}<\frac{1}{p}-\frac{1}{nr}$. Let $Q=[-1,1]^n$. Then there exists $w \in A_r$ such that 
\begin{equation}{\label{counterexample}}
\norm{f-f_Q}_{\strt{2.3ex}L^{q,\infty}\left(Q, \frac{w(x)\, dx}{w(Q)}\right)} \lesssim \left( \frac{1}{w(Q)}\int_{Q}|\nabla f(x)|^p\, w(x)\,dx\right)^{\frac{1}{p}}    
\end{equation}
 {\bf cannot} hold for every $f \in C_c^{\infty}(\R^n)$.
\end{proposition}

\begin{proof}
Suppose (\ref{counterexample}) holds for every $f \in C_c^{\infty}(\R^n)$. Let $B=B(0,1)$ denote the open unit ball in $\mathbb{R}^n$ centered at the origin. Then $$B \subset Q \subset \sqrt{n}B =:\widetilde{B}.$$
Since an $A_p$ weight satisfies the doubling condition $w(2Q) \lesssim w(Q)$, applying \eqref{weight 1,1 Poincare inequality}, for every $f \in C_c^{\infty}(\R^n)$, we have
\begin{align}
\|f-f_B\|_{\strt{2.3ex}L^{q,\infty}\left(B, \frac{w dx}{w(B)}\right)} & \le \|f-f_Q\|_{\strt{2.3ex}L^{q,\infty}\left(B, \frac{w dx}{w(B)}\right)}+|f_B-f_Q| \nonumber \\   
& \lesssim \|f-f_Q\|_{\strt{2.3ex}L^{q,\infty}\left(Q, \frac{w dx}{w(Q)}\right)}+\frac{1}{|Q|}\int_{Q}|f-f_Q| \nonumber \\ 
& \lesssim \Big(\frac{1}{w(Q)}\int_{Q}|\nabla f|^p w dx\Big)^{1/p} \nonumber \\
& \lesssim \Big(\frac{1}{w(\widetilde{B})}\int_{\widetilde{B}}|\nabla f|^p w dx\Big)^{1/p}. {\label{counterexample ball}}
\end{align}

Let $w(x)=|x|^{\delta - n}$ with $0 < \delta < nr$, it is well known that $w$ is an $A_r$ weight. Let $\varphi$ be a smooth function supported in $B(0,2)$ and taking the value $1$ on $B(0,1)$. For $0 < \epsilon < \frac{1}{2}$, we define $f_\epsilon(x)=\varphi(\frac{x}{\epsilon})$. Since $f_\epsilon$ vanishes on $E=B(0,1) \backslash B(0,\frac{1}{2})$  and $q>1$, applying ({\ref{counterexample ball}}),
\begin{align}
\|f_\epsilon\|_{\strt{2.3ex} L^{q,\infty}\left(B, \frac{w dx}{w(B)}\right)} & \le \|f_\epsilon - (f_\epsilon)_B\|_{\strt{2.3ex} L^{q,\infty}\left(B, \frac{w dx}{w(B)}\right)} +|(f_\epsilon)_B| \nonumber \\
& \le \|f_\epsilon - (f_\epsilon)_B\|_{\strt{2.3ex} L^{q,\infty}\left(B, \frac{w dx}{w(B)}\right)} + \|f_\epsilon-(f_\epsilon)_B\|_{\strt{2.3ex} L^{q,\infty}\left(E, \frac{w dx}{w(E)}\right)} \nonumber \\
& \lesssim \|f_\epsilon - (f_\epsilon)_B\|_{\strt{2.3ex}L^{q,\infty}\left(B, \frac{w dx}{w(B)}\right)} \nonumber \\
& \lesssim \Big(\frac{1}{w(\widetilde{B})}\int_{\widetilde{B}}|\nabla f_\epsilon|^p w dx\Big)^{1/p}, {\label{counterexample equal}}
\end{align}
where we use the fact that $w(B(0,1)) \simeq \frac{1}{\delta} \simeq w(E)$. Note that $f_\epsilon = 1$ on $B(0,\epsilon)$ and $|\nabla f_\epsilon| \lesssim_{\varphi} \frac{1}{\epsilon}\chi_{B(0,2\epsilon)}$, we have
$$\|f_\epsilon\|_{\strt{2.3ex} L^{q,\infty}\left(B, \frac{w dx}{w(B)}\right)} \ge \Big(\frac{w(B(0,\epsilon))}{w(B)}\Big)^{1/q} \simeq \epsilon^{\delta/q} $$
and 
\begin{equation}{\label{counterexample upper bound}}
\Big(\frac{1}{w(\widetilde{B})}\int_{\widetilde{B}}|\nabla f_\epsilon|^p w dx\Big)^{1/p} \lesssim \frac{1}{\epsilon}\Big(\frac{w(B(0,2\epsilon))}{w(\widetilde{B})}\Big)^{1/p} \simeq \epsilon^{\delta/p-1}.    
\end{equation}

Applying estimate ({\ref{counterexample equal}}), we have
\begin{equation}{\label{counterexample epsilon}}
\epsilon^{\delta(\frac{1}{q}-\frac{1}{p})+1} \lesssim 1.    
\end{equation}
Since $\frac{1}{q}-\frac{1}{p} < -\frac{1}{nr}$, we can  choose $\delta$ sufficiently close to $nr$, such that $\delta(\frac{1}{q} - \frac{1}{p}) < -1$. In this case, estimate (\ref{counterexample epsilon}) fails if we let $\epsilon$ tend to $0$, and as a result, we arrive at a contradiction.
\end{proof}

Now, we are going to construct the counterexample for the case $s > n$ and $1\le r \le p < rs$. Similarly to the previous case, applying the fractional Poincaré inequality (\ref{PSF SP}) with $p=1$ and \eqref{Ap 2}, for $0 < \alpha <1$  and $w \in A_p$,
\begin{align*}
\frac{1}{|Q|}\int_Q |f-f_Q| dx \lesssim  \ell(Q)^{\alpha}\left(\frac{1}{w(Q)}\int_{Q}|\mathcal D_Q^{\alpha} f(x)|^p w(x) dx\right)^{\frac{1}{p}}:=a(Q),
\end{align*}
where 
$$\mathcal{D}_Q^{\alpha}f(x):=\int_{Q}\frac{|f(x)-f(y)|}{|x-y|^{n+\alpha}} dy\qquad x\in Q
$$
is the local fractional derivative of order $\alpha$. From \cite[Lemma 3.2]{PR}, we know that $a \in SD_p^{s}(w)$ with $s=\frac{n}{\alpha}>n$. Then, applying Theorem {\ref{Selfimproving Ar}}, we obtain
\begin{equation*}
	\norm{f-f_Q}_{\strt{2.3ex} L^{p_r^{*},\infty}\left(Q, \frac{w(x)\, dx}{w(Q)}\right)} \lesssim \ell(Q)^{\alpha}\left(\frac{1}{w(Q)}\int_{Q}|\mathcal D_Q^{\alpha} f(x)|^p w(x) dx\right)^{\frac{1}{p}},
\end{equation*}
where $\frac{1}{p_r^{*}} = \frac{1}{p} - \frac{1}{sr}$. The following proposition yields the sharpness of the estimate (\ref{Thm Ar eq}) in this case.

\begin{proposition}\label{prop: counterexample fractional p=1} 
Suppose $1\le r\le p$ and $0 < \frac{1}{q}<\frac{1}{p}-\frac{\alpha}{nr}$. Let $Q=[-1,1]^n$. Then there exists $w \in A_r$ such that 
\begin{equation*}
\norm{f-f_Q}_{\strt{2.3ex} L^{q,\infty}\left(Q, \frac{w(x)\, dx}{w(Q)}\right)} \lesssim \left(\frac{1}{w(Q)}\int_{Q}|\mathcal D_Q^{\alpha} f(x)|^p w(x) dx\right)^{\frac{1}{p}}   
\end{equation*}
cannot hold for every $f \in C_c^{\infty}(\R^n)$.
\end{proposition}

\begin{proof}

Analogously to the proof of Proposition \ref{counterexample proposition}, we let $w(x)=|x|^{\delta - n}$ with $0 < \delta < nr$, and let $\varphi$ be a smooth function supported in $B(0,2)$ such that $\varphi \equiv 1$ on $B(0,1)$. For $0 < \epsilon < \frac{1}{2}$, we define $f_\epsilon(x)=\varphi(x/\epsilon)$. The desired result then follows by adapting the strategy used in that proposition; the only difference appears in the estimate ({\ref{counterexample upper bound}}). We decompose 
\begin{align}
& \int_{\widetilde{B}}|\mathcal{D}^{\alpha}_{\widetilde{B}} f_\epsilon|^p |x|^{\delta -n}dx \nonumber \\
\le & \int_{B(0,4\epsilon)}|\mathcal{D}^{\alpha}_{\widetilde{B}} f_\epsilon|^p |x|^{\delta -n}dx + \sum\limits_{k=2}^{\infty}\int_{B(0,2^{k+1}\epsilon) \backslash B(0,2^k\epsilon)}|\mathcal{D}^{\alpha}_{\widetilde{B}} f_\epsilon|^p |x|^{\delta -n}dx.  \nonumber  
\end{align}
For $x \in B(0, 4\epsilon)$,
\begin{align}
\mathcal{D}^{\alpha}_{\widetilde{B}} f_\epsilon(x) & = \int_{\widetilde{B}}\frac{|f_\epsilon(x)-f_{\epsilon}(y)|}{|x-y|^{n+\alpha}} dy \nonumber \\
& = \int_{B(x,4\epsilon)}\frac{|f_\epsilon(x)-f_{\epsilon}(y)|}{|x-y|^{n+\alpha}} dy + \int_{\widetilde{B} \backslash B(x,4\epsilon)}\frac{|f_\epsilon(x)-f_{\epsilon}(y)|}{|x-y|^{n+\alpha}} dy \nonumber \\
& \le \frac{1}{\epsilon}\|\nabla \varphi\|_{\strt{2.3ex}L^{\infty}(\R^n)}\int_{B(x,4\epsilon)}\frac{|x-y|}{|x-y|^{n+\alpha}} dy + \|\varphi \|_{\strt{2.3ex} L^{\infty}(\R^n)} \int_{\R^n \backslash B(x,4\epsilon)}\frac{1}{|x-y|^{n+\alpha}} dy\nonumber \\
& \le \frac{1}{\epsilon}\|\nabla \varphi\|_{\strt{2.3ex} L^{\infty}(\R^n)}\int_{B(0,4\epsilon)}|y|^{-n-\alpha+1}dy + \|\varphi\|_{\strt{2.3ex} L^{\infty}(\R^n)} \int_{\R^n \backslash B(0,4\epsilon)}|y|^{-n-\alpha} dy\nonumber\\
& \lesssim \epsilon^{-\alpha}. \nonumber 
\end{align}
For $x \in B(0,2^{k+1}\eps)\backslash B(0,2^k\epsilon)$ with $k \ge 2$, we have $f_\varepsilon(x)=0$ and $|x-y| \simeq |x| \simeq 2^k\epsilon$ for $y \in B(0,2\epsilon)$. Then, 
$$ \mathcal{D}^{\alpha}_{\widetilde{B}} f_\epsilon(x)  = \int_{B(0,2\epsilon)}\frac{|f_{\epsilon}(y)|}{|x-y|^{n+\alpha}} dy \lesssim \|\varphi \|_{\strt{2.3ex}L^{\infty}(\R^n)}(2^k\epsilon)^{-n-\alpha}(2\epsilon)^{n} \lesssim (2^k)^{-n-\alpha}\epsilon ^{-\alpha}.$$
Combining everything together,
\begin{align}
&\Big(\frac{1}{w(\widetilde{B})}\int_{\widetilde{B}}|\mathcal{D}^{\alpha}_{\widetilde{B}} f_\epsilon|^p |x|^{\delta -n}dx\Big)^{1/p} \nonumber \\
& \qquad \lesssim  \delta^{1/p}\Big(\epsilon^{-\alpha p}\int_{B(0,4\epsilon)} |x|^{\delta -n}dx + \epsilon^{-\alpha p}\sum\limits_{k=2}^{\infty}(2^k)^{-np-\alpha p}\int_{B(0,2^{k+1}\epsilon) \backslash B(0,2^k\epsilon)} |x|^{\delta -n}dx\Big)^{1/p} \nonumber \\
& \qquad \lesssim  \epsilon^{-\alpha}  \Big( \epsilon^{\delta} + \epsilon^{\delta}\sum\limits_{k=2}^{\infty}(2^k)^{-np-\alpha p + \delta}\Big)^{1/p} \\
& \qquad \lesssim \epsilon^{\delta/p - \alpha} \nonumber.
\end{align}
With this estimate, we can arrive at a contradiction by following the argument in Proposition {\ref{counterexample proposition}}.
\end{proof}

\section{Proofs of self-improving results to \texorpdfstring{$L^{p(1+\frac{1}{s}),\infty}$}{L{p(1+1/s),inf}} generalized Poincaré-type inequality}\label{Section 4}

In this section, we provide the proof of the two $L^{p(1+\frac{1}{s}),\infty}$ generalized Poincaré-type inequalities with different assumptions. First, we provide the proof of Theorem \ref{Selfimproving Ainfty}, assuming $w\in A_\infty$.

\begin{proof}[Proof of Theorem \ref{Selfimproving Ainfty}]

The proof follows a similar strategy to that of Theorem \ref{Selfimproving Ar}; hence, we focus only on the necessary modifications.
Starting from (\ref{middle step}), using Minkowski's integral inequality, we obtain
\begin{equation*}
	\gamma t < \frac{1}{|R_i|}\int_{R_i}|f-f_{R_i}|  \leq \frac{1}{|R_i|}\int_{R_i} |f-f_{Q}| + |f_Q-f_{R_i}|
\leq  2\frac{1}{|R_i|}\int_{R_i}|f-f_{Q}|.
\end{equation*}
Hence,
\begin{equation*}
	\left| \bigcup_iR_i\right| =\sum_i |R_i| < \frac{2}{\gamma t} \sum_i \int_{R_i}|f-f_{Q}|
\leq \frac{2}{\gamma t} \int_{Q}|f-f_{Q}| \leq \frac{2}{\gamma t} a(Q)|Q|.
\end{equation*}
Then, 
\begin{align*}
	\sum_j w(B_j)  \leq &  \|a\|^p \left( \frac{1}{\gamma t}\right)^p \, \left( \frac{2}{\gamma t} a(Q)\right )^{ \frac{p}{s} } w(Q) a(Q)^p\\
= & \|a\|^p 2^{\frac{p}{s}}  \,  \left( \frac{1}{\gamma t} a(Q)\right )^{ p+\frac{p}{s} } w(Q).
\end{align*}

We proceed as before, combining everything together, and we get
\begin{align*}
 w(\Omega_{\kappa t}) & \leq c_1 e^{-\frac{c_2}{ [w]_{A_{\infty}}\gamma } }w(\Omega_t) + 
\|a\|^p 2^{\frac{p}{s}}  \,  \left( \frac{1}{\gamma t} a(Q)\right )^{ p+\frac{p}{s} } w(Q),
\end{align*}
and then 
\begin{align*}
(\kappa t)^{p+\frac{p}{s}} w(\Omega_{\kappa t}) & \leq c_1 e^{-\frac{c_2}{ [w]_{A_{\infty}}\gamma } }\, (\kappa t)^{p+\frac{p}{s}}w(\Omega_t) + 
\|a\|^p \left( \frac{\kappa}{\gamma }\right)^{p+\frac{p}{s}}  2^{\frac{p}{s}} 
w(Q) a(Q)^{p+\frac{p}{s}}.
\end{align*}

Defining
\begin{equation*}
	\varphi(N) = \sup_{0<t\leq N} t^{p + \frac{p}{s}} w(\Omega_t)
\end{equation*}
and choosing $\gamma$ satisfying
\begin{equation*}
	c_1 \kappa^{p+\frac{p}{s}} e^{-\frac{c_2}{\gamma [w]_{A_{\infty}}}}=\frac 12,
\end{equation*}
we obtain that 
\begin{equation*}
	\varphi(N) \leq c_n^{p+\frac{p}{s}}[w]_{A_\infty}^{p+\frac{p}{s}} \Big(p(1+\frac{1}{s})\Big)^{p+\frac{p}{s}}\|a\|^{p} w(Q) a(Q)^{p+\frac{p}{s}}. 
\end{equation*}
Finally, letting $N \to \infty$ completes the proof. 
\end{proof}

Now, we turn to the proof of Theorem \ref{Thm1}, which covers the case of general weights.

\begin{proof}[Proof of Theorem \ref{Thm1}] 
Let $0<p<\infty$. Fix a cube $Q$ and let $t>a(Q)$. Observe that by \eqref{Thm1 1} we have,
	\begin{equation*}
		\frac{1}{|Q|}\int_Q |f(x)-f_Q| dx \le a(Q) < t.
	\end{equation*}
We consider the local Calder\'on--Zygmund decomposition of the function $|f-f_Q|$ relative to the cube $Q$ at height $t$. Then, there is a family of dyadic pairwise disjoint subcubes $\{Q_j\}_j$ with respect to $Q$, which satisfy the following properties. For each $j$,
    \begin{equation}\label{Thm1 P CZ1}
	t<\frac{1}{|Q_{j}|}\int_{Q_{j}} |f(x)-f_Q| dx \le 2^n t,
	\end{equation}   
    and for almost every $x\in Q\setminus \bigcup_j Q_{j}$, 			
    \begin{equation*}
		|f(x)-f_Q| \le t.
	\end{equation*}
    From  the first initial hypothesis \eqref{Thm1 1}  and (\ref{Thm1 P CZ1}), 		\begin{equation}\label{Thm1 P CZ2}
		\frac{\left| \bigcup_j Q_j \right| }{|Q|}	\le \frac{1}{t|Q|}\int_{Q}|f-f_Q| \, dx \le \frac{a(Q)}{t}.
	\end{equation}
Let $\kappa=2^n + 1$. Since \eqref{Thm1 P CZ1} implies that $|f_{Q_j}-f_Q| \le 2^nt$, 
we have, 
\begin{align*}
w\left(  \left\lbrace x\in Q : |f(x)-f_Q| > \kappa t \right\rbrace \right)  & \le  \sum_j w\left(  \left\lbrace x\in Q_{j} : |f(x)-f_{Q_{j}}|+|f_{Q_j}-f_Q| > \kappa t \right\rbrace\right) \\
& \le  \sum_j w\left(  \left\lbrace x\in Q_{j} : |f(x)-f_{Q_{j}}| > t \right\rbrace\right). 
\end{align*}
	We can estimate the sum using the second hypothesis   \eqref{Thm1 2} at each cube $Q_j$ and then the fact that the functional $a$ satisfies the $SD_p^s(w)$ condition, 
		\begin{align*}
			w\left(  \left\lbrace x\in Q : |f(x)-f_Q| > \kappa t \right\rbrace \right)   \le & \sum_j w\left(  \left\lbrace x\in Q_{j} : |f(x)-f_{Q_{j}}| > t \right\rbrace\right)   \\
			\le & K^p t^{-p}  \sum_j    a(Q_{j})^p w(Q_{j})   \\
			\le &  K^p t^{-p} \left\| a \right\| ^p \left( \frac{\left| \bigcup_j Q_j \right|}{|Q|} \right)^\frac{p}{s} a(Q)^p w(Q)\\
			\le & K^p \left\| a \right\| ^p \left( \frac{a(Q)}{t}\right)^{p+\frac{p}{s}}  w(Q)\\
		\end{align*}
		where in the last inequality, we used \eqref{Thm1 P CZ2}.  Combining these estimates, we have proved 
		\begin{equation*}
			t \, w\left(  \left\lbrace x\in Q : |f(x)-f_Q| > \kappa t \right\rbrace \right)  ^\frac{1}{p+\frac{p}{s}}  \le  \left( K \left\| a \right\| \right)^{\frac{s}{s+1}}  w(Q)^\frac{1}{p+\frac{p}{s}} a(Q) 
		\end{equation*}
		for all $t>a(Q)$. This is equivalent to saying that 
		\begin{equation}\label{Thm1 P main estimation 2}
			\sup_{t>c_n a(Q)} t \left( \frac{1}{w(Q)} w\left(  \left\lbrace x\in Q : |f(x)-f_Q| >  t \right\rbrace \right) \right) ^\frac{1}{p+\frac{p}{s}}  \le  c_n \left( K \left\| a \right\| \right)^{\frac{s}{s+1}} a(Q) .
		\end{equation}
		On the other hand, 
		\begin{equation*}
			\sup_{t\le c_n a(Q) } t \left( \frac{1}{w(Q)} w\left(  \left\lbrace x\in Q : |f(x)-f_Q| >  t \right\rbrace \right) \right) ^\frac{1}{p+\frac{p}{s}}  \le  c_n a(Q) .
		\end{equation*}
		Hence,
		\begin{align*}
			\left\| f - f_Q \right\| _{L^{p+\frac{p}{s},\infty}\left( Q, \frac{w(x)dx}{w(Q)}\right) } 
			= & \max \left\lbrace \sup_{t\le c_n a(Q) } \cdots ,  \sup_{t>c_n a(Q)} \cdots   \right\rbrace \\
			\le & \max \left\lbrace c_n a(Q)  ,  c_n \left( K \left\| a \right\| \right)^{\frac{s}{s+1}}   a(Q)  \right\rbrace \\
			= & c_n \max \left\lbrace 1   ,  \left( K \left\| a \right\| \right)^{\frac{s}{s+1}}     \right\rbrace a(Q) .
		\end{align*}	
        This concludes the proof.
\end{proof}

As an immediate consequence of the argument, one also obtains a self-improvement along the Lorentz scale: under the assumptions of Theorem \ref{Thm1}, using \eqref{Lorentz inclusion} we have 
\begin{equation}\label{Thm1 Lpq 3}
		\left\| f - f_Q\right\|_{\strt{2.3ex} L^{p,q}\left( Q, \frac{w(x)\,dx}{w(Q)}\right)} \le c_n \left( \frac{p}{q}\right)^{\frac{1}{q}} \max\left\{ 1 , \left( K \left\| a \right\| \right)^{\frac{s}{s+1}} \right\} (1+s)^{\frac{1}{q}}\, a(Q),
	\end{equation}
for every $1\le q<\infty$. In particular, within this framework, the weak
estimate \eqref{Thm1 2} is, up to quantitative constants, equivalent to the family of strong Lorentz estimates \eqref{Thm1 Lpq 3}; this is noteworthy since the endpoint $q=1$ is a priori much stronger than $q=\infty$.

\section{Applications to degenerate Poincaré-Sobolev inequalities}\label{Section 5}

In this section, we investigate the Lorentz-type Poincar\'e--Sobolev inequalities stated in Section \ref{Section PS}. The proofs of strong-type results Corollary \ref{thm:PS-Lorentz-strong} and Corollary \ref{prop:pairs} have already been discussed before Proposition {\ref{counterexample proposition}}. To prove Corollaries \ref{thm:PS-Lorentz-weak} and \ref{thm:PS-exp}, we will need two auxiliary lemmas. The first one is a geometric property satisfied by Lorentz quasi-norms; it is in the spirit of the well-known result from \cite{CHK}.

\begin{lemma}\label{Lemma Lorentz}
		Let $(X,\mu)$ be a measure space. Let $1\le p,q<\infty$ and $\alpha\ge \max \{p,q\}$. Then we have		
        \begin{equation*}
			\sum_j \left\| f \right\|_{\strt{2.3ex} L^{p,q}\left(E_j , \mu\right)}^\alpha \le \left\| f \right\|_{\strt{2.3ex} L^{p,q}\left(\cup_j E_j , \mu \right)}^\alpha
		\end{equation*}
		for each pairwise disjoint family $\{E_j\}_j$ of $\mu$-measurable sets.
	\end{lemma}
	
	\begin{proof}
		Let $\{E_j\}_j$ be a pairwise disjoint family of $\mu$-measurable sets.  Since $\frac{\alpha}{q}\ge 1$, we can use Minkowski's integral inequality,
		
	\begin{align*}
			\left( \sum_j \left\| f \right\|_{\strt{2.3ex} L^{p,q}\left(E_j , \mu\right)}^\alpha \right)^\frac{q}{\alpha}  = & \left( \sum_j  \left( p \int_0^\infty t^q \mu\left(\{ x\in E_j : |f(x)| > t \} \right)^\frac{q}{p} \frac{dt}{t} \right)^\frac{\alpha}{q} \right)^\frac{q}{\alpha} \\
			= & \left\| \left\{ p \int_0^\infty t^q \mu \left(\{ x\in E_j : |f(x)| > t \} \right)^\frac{q}{p} \frac{dt}{t} \right\}_j \right\|_{\ell^\frac{\alpha}{q}}\\
			\le &  p \int_0^\infty t^q \left\| \left\{\mu\left(\{ x\in E_j : |f(x)| > t \} \right)^\frac{q}{p} \right\}_j \right\|_{\ell^\frac{\alpha}{q}}\frac{dt}{t} \\
			= & p \int_0^\infty t^q \left\| \left\{\mu\left(\{ x\in E_j : |f(x)| > t \} \right) \right\}_j \right\|_{\ell^\frac{\alpha}{p}} ^\frac{q}{p} \frac{dt}{t} \\
			\le & p \int_0^\infty t^q \left\| \left\{\mu\left(\{ x\in E_j : |f(x)| > t \} \right) \right\}_j \right\|_{\ell^1} ^\frac{q}{p} \frac{dt}{t} \\
			= & p \int_0^\infty t^q  \mu \left(\left\{ x\in 	\cup_j E_j : |f(x)| > t \right\} \right) ^\frac{q}{p} \frac{dt}{t} \\
			= & \left\| f \right\|_{\strt{2.3ex} L^{p,q}\left(\cup_j E_j , \mu \right)}^q,
		\end{align*}
		where, in the two inequalities above, we have used that $\frac{\alpha}{q}\ge 1$ and $\frac{\alpha}{p}\ge 1$ respectively.
	\end{proof}

The next lemma shows that the functional arising in our Lorentz-type estimates satisfies the corresponding geometric condition under the natural assumptions.

\begin{lemma}\label{Lemma Lpq SDp}
	Let $w\in A_r$ with $1\le r \le p$. For $1 \le M \le \infty$, we define $p_{M}^*$ by the relation
	\begin{equation*}
		\frac{1}{p}- \frac{1}{p_{M}^*} = \frac{1}{n}\frac{1}{rM}.
	\end{equation*}
We consider the following functional over cubes:	
	\begin{equation*}
		a(Q)= \ell(Q) \left\| g \right\|_{\strt{2.3ex} L^{p,q} \left( Q, \frac{w(x)\, dx}{w(Q)}\right)},
	\end{equation*}
	where $q\le p_{M}^*$. Then,
\begin{enumerate}
 	\item If $M>1$, then $a \in SD_{p_{M}^*}^{nM'}(w)$ with $\|a\|\le [w]_{A_r}^{\frac{1}{nrM}}$.
 	\item If $M=1$, then $p_{M}^*=p_r^{*}$  and $a\in D_{p^*_{r}}(w)$ with $\|a\|\le [w]_{A_r}^{\frac{1}{nr }}$.
 \end{enumerate}
\end{lemma}

\begin{proof}
Let $Q$ be a cube and let $\{Q_j\}_j$ be a collection of pairwise disjoint subcubes of $Q$. Invoking \eqref{Ap 2}, which is a consequence of the $A_r$ condition, we obtain
		\begin{align*}
			\sum_j a(Q_j)^{p_{M}^*} w(Q_j)  = & \sum_j \ell(Q_j)^{p_{M}^*} \left\| g \right\|_{L^{p,q} \left( Q_j, \frac{wdx}{w(Q_j)}\right)} ^{p_{M}^*} w(Q_j) \\
			= & \sum_j |Q_j|^\frac{p_{M}^*}{n} \left(\frac{1}{w(Q_j)} \right)^{\frac{p_{M}^*}{nrM}} \left\| g \right\|_{\strt{2.3ex} L^{p,q} \left( Q_j, w\right)} ^{p_{M}^*}  \\
			= & \sum_j |Q_j|^{\frac{p_{M}^*}{n M'}}\left(\frac{|Q_j|^r}{w(Q_j)} \right)^{\frac{p_{M}^*}{nrM}} \left\| g \right\|_{\strt{2.3ex} L^{p,q} \left( Q_j, w\right)} ^{p_{M}^*}  \\
			\le & \left([w]_{A_r}\frac{|Q|^r}{w(Q)} \right)^{\frac{p_M^{*}}{nrM}} \sum_j |Q_j|^{\frac{p_{M}^*}{n M'}}\left\| g \right\|_{\strt{2.3ex} L^{p,q} \left( Q_j, w\right)} ^{p_{M}^*}. \\
            \le & \left([w]_{A_r}\frac{|Q|^r}{w(Q)} \right)^{\frac{p_M^{*}}{nrM}} \left( \sup_j |Q_j|^{\frac{p_{M}^*}{n M'}}\right) \sum_j \left\| g \right\|_{\strt{2.3ex} L^{p,q} \left( Q_j, w\right)} ^{p_{M}^*}.
		\end{align*}
		
		By hypothesis we have $q\le p^*_M$, so we can use Lemma \ref{Lemma Lorentz} with $\alpha=p^*_M$,
		\begin{align*}
			\sum_j a(Q_j)^{p_{M}^*} w(Q_j)  \le & \left([w]_{A_r}\frac{|Q|^r}{w(Q)} \right)^{\frac{p_M^{*}}{nrM}} \left| \bigcup_j Q_j\right|^{\frac{p_M^{*}}{n M'}}   \left\| g \right\|_{\strt{2.3ex}L^{p,q} \left( Q, w\right)} ^{p_M^{*}}  \\
			= & [w]_{A_r}^{\frac{p_M^{*}}{nrM}}  \left( \frac{\left| \bigcup_j Q_j\right|}{|Q|} \right)^{\frac{p_M^{*}}{n M'}} \ell(Q)^{p_M^{*}} \left\| g \right\|_{\strt{2.3ex}L^{p,q} \left( Q, \frac{wdx}{w(Q)}\right)}^{p_M^{*}} w(Q)  \\
			= &  [w]_{A_r}^{\frac{p_{M}^*}{nrM}}  \left( \frac{\left| \bigcup_j Q_j\right|}{|Q|} \right)^{\frac{p_{M}^*}{n M'}} a(Q)^{p_{M}^*} w(Q).
		\end{align*}
        This proves the first assertion for $1<M\le\infty$. In particular, when $M=\infty$, one has $p_M^*=p$, $M'=1$, and $[w]_{A_r}^{1/(nrM)}=1$. When $M=1$, the same computation gives $p_M^*=p_r^*$ and $1/M'=0$, hence 
        \begin{equation*}
            \sum_j a(Q_j)^{p_r^*}w(Q_j) \le [w]_{A_r}^{\frac{p_r^*}{nr}} a(Q)^{p_r^*}w(Q),
        \end{equation*}
        which is precisely the $D_{p_r^*}(w)$ condition with $\|a\|\le [w]_{A_r}^{1/(nr)}$.
	\end{proof}

The weighted Poincaré-Sobolev estimate \eqref{self improving example 1} provides a convenient example to compare the self-improvement phenomena associated with the conditions $a\in D_{p_r^{*}}(w)$ and $a\in SD_{p}^{n}(w)$. Denote by $a(Q)$ the right-hand side functional in \eqref{self improving example 1}, and for $M\ge 1$ we consider $p_M^{*}$ defined in Lemma~\ref{Lemma Lpq SDp}. We have proved that $a\in D_{p_r^{*}}(w)$ with $\|a\|\le [w]_{A_r}^{1/(nr)}$ and $a\in SD_{p_M^{*}}^{nM'}(w)$ with $\|a\|\le [w]_{A_r}^{1/(nrM)}$ for every $M>1$. Consequently, whether one applies the self-improving result under the $D_{p_r^{*}}(w)$ condition (Theorem \ref{Thm 1.6 CP}) or the one under the $SD_{p_M^{*}}^{nM'}(w)$ condition (Theorem \ref{Selfimproving Ar}), one recovers the same conclusion, namely \eqref{self improving example 1}. This shows the compatibility between the $D_p(w)$ and $SD_p^s(w)$ conditions in this natural model example.

Now, we apply the previous result to provide the proof of Corollary \ref{thm:PS-Lorentz-weak}.

\begin{proof}[Proof of Corollary \ref{thm:PS-Lorentz-weak}]
	Since $1 \le r \le p$ and $p_r^{*}>1$, we have $A_r\subset A_p = A_{p,p_r^{*}}$. Combining the $(1,1)$ Poincaré inequality \eqref{I1 P11} with the characterization of the $A_{p,q}$ condition given in \eqref{Apq}, applied with $g=|\nabla f|$, we obtain
		\begin{equation*}
			\frac{1}{|Q|}\int_Q |f(x)-f_Q|dx \le c_n  [w]_{A_{p,p_r^{*}}}^\frac{1}{p} \ell (Q)\left\| \nabla f \right\|_{\strt{2.3ex} L^{p,p_r^{*}}\left(Q, \frac{w(x)dx}{w(Q)}\right)}
		\end{equation*}
		for each cube $Q$. Now we consider the functional
		\begin{equation*}
			a(Q)= c_n [w]_{A_{p,p_r^{*}}}^\frac{1}{p} \ell(Q)\left\| \nabla f \right\|_{\strt{2.3ex} L^{p,p_r^{*}}\left(Q, \frac{w(x)dx}{w(Q)}\right)}.
		\end{equation*}	
		By Lemma \ref{Lemma Lpq SDp} (with $M=1$) we have $a\in D_{p_r^{*}}(w)$ and $\|a\|\le [w]_{A_r}^{1/(nr)}$. Therefore, we can apply Theorem \ref{Thm 1.6 CP} and this yields \eqref{eq:PS-Lorentz-weak}.
\end{proof} 

Let us emphasize that the previous argument is specific to the Lorentz norm $\|\cdot\|_{L^{p,p_r^{*}}}$ on the right-hand side. Indeed, in this setting we have $p_r^*=p_M^*$ when $M=1$, and this endpoint case corresponds to the condition $D_{p_M^*}(w)$ rather than to an $SD_p^s(w)$-type condition. For this reason, in the proof of Corollary \ref{thm:PS-Lorentz-weak}, we have used Theorem \ref{Thm 1.6 CP} instead of Theorem \ref{Selfimproving Ar}. By contrast, as we discussed before the proof of Corollary \ref{thm:PS-Lorentz-weak}, if we consider the self-improving result with the norm $\|\cdot\|_{L^{p}}$ on the right-hand side, then Theorem \ref{Selfimproving Ar} and Theorem \ref{Thm 1.6 CP} will lead to the same target exponent, since $(p_M^*)_{r,nM'}^*=p_r^*$.

To conclude this section, we will apply Theorem \ref{Selfimproving Ar} to prove Corollary \ref{thm:PS-exp}, providing an $\exp L$-type Poincaré-Sobolev inequality in the case of $p\ge nr$.

\begin{proof}[Proof of Corollary \ref{thm:PS-exp}]
Fix $1 \le q<\infty$. Choose $M>1$ so that $q\le p_M^{*}$, where $p_M^{*}$ is the exponent from Lemma \ref{Lemma Lpq SDp}. Consider the cube functional
	\begin{equation*}
			a(Q)= [w]_{A_{p,q}}^\frac{1}{p} \ell(Q)\left\| \nabla f \right\|_{\strt{2.3ex} L^{p,q}\left(Q, \frac{w(x)dx}{w(Q)}\right)}.
	\end{equation*}
By Lemma \ref{Lemma Lpq SDp} we have $a\in SD_{p_M^{*}}^{\,nM'}(w)$ with
$\|a\|\le [w]_{A_r}^{1/(nrM)}$. Since $p\ge nr$, it follows that $p_M^{*}\ge rnM'$. Therefore, the second part of Theorem \ref{Selfimproving Ar} applies and yields \eqref{eq:PS-exp}.
\end{proof}

\section{Applications to fractional Poincaré-Sobolev inequalities }\label{Section 6}

In this section, we investigate the fractional Poincaré-Sobolev inequalities. In order to consider two-weight $A_p$ conditions, we need the following result adapted from \cite{MPW}.

\begin{theorem}\label{thm:MPW23} 
    Let $0<\delta<1$, $1\le p<\infty$, $f\in L^1_{\operatorname{loc}}(\R^n)$, and let $(u,v)\in A_p$. Then there exists a dimensional constant $c_n>0$ such that
    \begin{equation}\label{MPW eq}
        \frac{1}{|Q|}\int_Q |f(x)-f_Q|\,dx \le c_n [u,v]_{A_p}^{\frac{1}{p}}(1-\delta)^{\frac{1}{p}}\ell(Q)^\delta \left(\frac{1}{u(Q)}\int_Q\int_Q \frac{|f(x)-f(y)|^p}{|x-y|^{n+\delta p}}\,dy\,v(x)\,dx\right)^{\frac{1}{p}}
    \end{equation}
    for every cube $Q\subset \R^n$.
\end{theorem}

\begin{proof}
Assume first that $0<\delta\le \frac12$. By H\"older's inequality, we have
\begin{align*}
    \frac{1}{|Q|}\int_Q |f-f_Q|\,dx
    &\le \frac{1}{|Q|^2}\int_Q\int_Q |f(x)-f(y)|\,dy\,dx \\
    &\le C_n\ell(Q)^\delta \frac{1}{|Q|}\int_Q \left(\int_Q \frac{|f(x)-f(y)|^p}{|x-y|^{n+\delta p}}\,dy\right)^{\frac{1}{p}} dx \\
    &\le C_n [u,v]_{A_p}^{\frac{1}{p}}\ell(Q)^\delta\left(\frac{1}{u(Q)}\int_Q\int_Q \frac{|f(x)-f(y)|^p}{|x-y|^{n+\delta p}}\,dy\,v(x)\,dx\right)^{\frac{1}{p}} .
\end{align*}
Since $(1-\delta)^{\frac{1}{p}}\ge 2^{-\frac{1}{p}}$ in this range, this gives the claimed bound.

   It remains to consider the case $\frac12\le \delta<1$. We follow the proof of the one-weight version in \cite[Corollary 5.2]{MPW}. Fix a cube $Q\subset \R^n$ and let $0\le \varepsilon\le \delta$. By \eqref{PSF SP} applied with parameter $\delta-\varepsilon$, we have
\begin{equation*}
    \frac{1}{|Q|}\int_Q |f(x)-f_Q|\,dx \le C_n(1-\delta+\varepsilon)\ell(Q)^{\delta-\varepsilon}\frac{1}{|Q|}\int_Q\int_Q \frac{|f(x)-f(y)|}{|x-y|^{n+\delta-\varepsilon}}\,dy\,dx .
\end{equation*}
If $p=1$, the conclusion follows from this estimate with $\varepsilon=0$ and the definition of the two-weight $A_1$ condition. We may therefore assume that $p>1$. For $0<\varepsilon\le \delta$, H\"older's inequality in the $y$ variable gives, for every $x\in Q$,
\begin{equation*}
    \int_Q \frac{|f(x)-f(y)|}{|x-y|^{n+\delta-\varepsilon}}\,dy \le C_n\frac{\ell(Q)^\varepsilon}{\varepsilon^{\frac{1}{p'}}}\left(\int_Q \frac{|f(x)-f(y)|^p}{|x-y|^{n+\delta p}}\,dy\right)^{\frac{1}{p}} .
\end{equation*}
Hence, applying H\"older's inequality again, we obtain
\begin{align*}
    \frac{1}{|Q|}\int_Q |f-f_Q|\,dx
    &\le C_n\ell(Q)^\delta\frac{1-\delta+\varepsilon}{\varepsilon^{\frac{1}{p'}}}\frac{1}{|Q|}\int_Q \left(\int_Q \frac{|f(x)-f(y)|^p}{|x-y|^{n+\delta p}}\,dy\right)^{\frac{1}{p}} dx \\
    &\le C_n [u,v]_{A_p}^{\frac{1}{p}}\ell(Q)^\delta\frac{1-\delta+\varepsilon}{\varepsilon^{\frac{1}{p'}}}\left(\frac{1}{u(Q)}\int_Q\int_Q \frac{|f(x)-f(y)|^p}{|x-y|^{n+\delta p}}\,dy\,v(x)\,dx\right)^{\frac{1}{p}} .
\end{align*}
Choosing $\varepsilon=1-\delta$ yields $\frac{1-\delta+\varepsilon}{\varepsilon^{\frac{1}{p'}}}=2(1-\delta)^{\frac{1}{p}}$. This proves the desired estimate also for $\frac12\le \delta<1$.
\end{proof}

\begin{remark}
In fact, the previous result is a small improvement of \cite[Corollary 5.2]{MPW}, where an extra factor $\delta^{-1/p'}$ appears on the right-hand side. The argument of considering the cases $\delta \in (0, \tfrac{1}{2})$ and $\delta \in [\tfrac{1}{2}, 1)$ separately to eliminate the factor $\delta^{-1/p'}$ can be found in the one-weight setting in \cite[p. 71]{ThesisIker}.
\end{remark}

We will use the following result from \cite{CMPR}. Although it is stated in the product space setting, the argument for cubes is identical; therefore, we state it below in that form.
  
 \begin{lemma}{\cite[Lemma 6.2]{CMPR}}{\label{CMPR Lemma}} 
 Let $w\in A_r$ with $1\le r\le  p$, let $0<\delta<1$ and let $a$ be defined by
 	 \begin{equation}\label{a2}
 		a(Q)=\ell(Q)^\delta \left( \frac{1}{w(Q)} \int_Q \int_Q g(x,y)\, dy \, dx \right)^\frac{1}{p}.
 	\end{equation}
 	
 	Let $1 \le M \le \infty$ and let $p^*_{r, M,\delta}$ be defined by the condition
 	\begin{equation}\label{p*M11}
		\frac{1}{p}-\frac{1}{p^*_{r, M,\delta}}=\frac{\delta}{n } \frac{1}{r M}.
	\end{equation}
 	
 	We have:
 	\begin{enumerate}
 		\item If $M>1$, then $a\in SD_{p^*_{r, M,\delta}}^{\frac{nM'}{\delta}}(w)$ with $\|a\|\le [w]_{A_r}^{\frac{\delta}{nr M}}$.
 		\item If $M=1$, then $a\in D_{p^*_{r, 1, \delta}}(w)$ with $\|a\|\le [w]_{A_r}^{\frac{\delta}{nr }}$.
 	\end{enumerate}
 \end{lemma}
 
 We shall prove the two-weight version stated in Corollary \ref{thm:fractional-PS-pairs}, since the one-weight statement in Corollary \ref{thm:fractional-PS} follows by taking $w=u=v$.

\begin{proof}[Proof of Corollary \ref{thm:fractional-PS-pairs}]
We start with the subcritical estimate \eqref{MPW eq}. Fix a cube $Q$ and set $a(Q)$ to be the right-hand side of \eqref{MPW eq}, namely
\begin{equation*}
	a(Q):= c_n (1-\delta)^\frac{1}{p} [u,v]_{A_p}^\frac{1}{p} \ell(Q)^\delta \left( \frac{1}{u(Q)} \int_Q \int_Q \frac{|f(x)-f(y)|^p}{|x-y|^{n+\delta p }} dy\, v(x) dx\right)^\frac{1}{p}.
\end{equation*}
Then \eqref{MPW eq} yields the starting inequality
\begin{equation*}
	\frac{1}{|Q|}\int_Q |f-f_Q|\,dx \le a(Q),
\end{equation*}
for every cube $Q$. Since $a(Q)$ has the form \eqref{a2}, Lemma \ref{CMPR Lemma} implies that for each $1<M\le\infty$ the functional $a$ verifies $SD_{p^{*}_{r,M,\delta}}^{\,nM'/\delta}(u)$ with $\|a\|\le [u]_{A_r}^{\delta/(nrM)}$,  while for $M=1$ it satisfies the corresponding $D_{p^{*}_{r,1,\delta}}(u)$ condition with $\|a\|\le [u]_{A_r}^{\delta/(nr)}$, where the exponent $p^{*}_{r,M,\delta}$ is defined by \eqref{p*M11}. We may now apply Theorem~\ref{Selfimproving Ar} (for any $1<M\le\infty$) or Theorem \ref{Thm 1.6 CP} with $M=1$, to obtain
\begin{equation*}
	\|f-f_Q\|_{\strt{2.3ex} L^{p^*_{r,\delta},\infty}\left(Q,\frac{u(x)\,dx}{u(Q)}\right)} \le c_n p^*_{r, \delta}  [u]_{A_\infty} [u]_{A_r}^{\frac{\delta }{nr}}   \,a(Q),
\end{equation*}
for every cube $Q$. Finally, applying the truncation method for fractional seminorms (Theorem~\ref{Truncation method}) with $q=p^*_{r,\delta}$, $d\mu(x)=\frac{u(x)\,dx}{u(Q)}$ and $K(y,z)=\frac{\ell(Q)^{\delta p}}{u(Q)}\frac{v(y)}{|y-z|^{n+\delta p}}$, we obtain the desired result. 

The endpoint estimate \eqref{exp result fractional} follows directly from Theorem \ref{Selfimproving Ar} in the critical/supercritical regime, together with the fact that $a \in SD_p^{\,n/\delta}(u)$. This completes the proof.
\end{proof}

\ed

\begin{remark}{\label{Fractional PS sharpness}}
Using the counterexample we constructed in Proposition \ref{prop: counterexample fractional p=1}, we can see that Corollary \ref{thm:fractional-PS} is optimal within the range satisfying both conditions $ 1 \le r \le p $ and $ \frac{\delta}{n}p < r \le 1 + \frac{\delta}{n}p$; the details are left to the reader. 
\end{remark}

\section{Weak Poincaré-Sobolev inequality and applications to \texorpdfstring{$A_1$}{A1} weights}\label{Section Weak}

We present the proof of Theorem \ref{Thm PS 4} and Corollary \ref{Thm PS 5}. We will use the following weak Fefferman--Stein-type isoperimetric inequality.

\begin{theorem}\label{Thm FPW}
    Let $n\ge 2$ and let $\mu$ be a non-negative Borel measure. Then, there exists a dimensional constant $c_n>0$ such that for all Lipschitz functions $f$ we have
        \begin{equation}\label{Thm FPW Eq}
            \left\| f-f_Q\right\|_{\strt{2.3ex} L^{n',\infty}(Q,\mu)}
            \le c_n \int_Q |\nabla f(x)| M^c(\chi_Q\mu)(x)^\frac{1}{n'} dx,
        \end{equation}
    for each cube $Q$.
\end{theorem}

The previous result was first established in \cite{FPW} in a different and more general setting. It was later considered in the Euclidean setting and extended in \cite{PR}, while a different proof, avoiding Minkowski's inequality, was given in \cite{PR2}. 

\begin{proof}[Proof of Theorem \ref{Thm PS 4}]
By the definition of the Lorentz quasi-norm and the pointwise equivalence of the Hardy-Littlewood maximal operator $M$ and the centered Hardy-Littlewood maximal operator $M^c$, it suffices to prove that there exists $c_{n,p}>0$ such that 
\begin{equation}\label{Thm PS 4 normalized Eq}
    \left\| f-f_Q\right\|_{\strt{2.3ex} L^{p^*,\infty}\left(Q,\frac{w(x)dx}{w(Q)}\right)} \le c_{n,p} w(Q)^\frac{1}{n} \left\| \left| \nabla f \right| \frac{M(w\chi_Q)^\frac{1}{n'}}{w} \right\|_{\strt{2.3ex} L^{p,\infty}\left(Q,\frac{w(x)dx}{w(Q)}\right)}.
\end{equation}
For simplicity, set
\begin{equation*}
    A:=w(Q)^\frac{1}{n} \left\| \left| \nabla f \right| \frac{M(w\chi_Q)^\frac{1}{n'}}{w} \right\|_{\strt{2.3ex}L^{p,\infty}\left(Q,\frac{w(x)dx}{w(Q)}\right)}.
\end{equation*}
If $A=0$, the desired estimate is trivial. Hence, we may assume that $A>0$. Let $t_0= C_{n,p}A$, where $C_{n,p}>0$ is a sufficiently large constant to be fixed below. Since
\begin{align*}
    \left\| f-f_Q\right\|_{\strt{2.3ex}L^{p^*,\infty}\left(Q,\frac{w(x)dx}{w(Q)}\right)} = \max \bigg\{& \sup_{0<t\le t_0} t \left(\frac{w(\{x\in Q:|f(x)-f_Q|>t\})}{w(Q)}\right)^\frac{1}{p^*}, \\
    & \sup_{t_0\le t} t \left(\frac{w(\{x\in Q:|f(x)-f_Q|>t\})}{w(Q)}\right)^\frac{1}{p^*} \bigg\},
\end{align*}
we consider two cases.

\textbf{Case 1:}
\begin{equation*}
    \sup_{0<t\le t_0} t \left(\frac{w(\{x\in Q:|f(x)-f_Q|>t\})}{w(Q)}\right)^\frac{1}{p^*} \ge \sup_{t_0\le t} t \left(\frac{w(\{x\in Q:|f(x)-f_Q|>t\})}{w(Q)}\right)^\frac{1}{p^*}.
\end{equation*}
In this case, we have
\begin{equation*}
    \left\| f-f_Q\right\|_{\strt{2.3ex} L^{p^*,\infty}\left(Q,\frac{w(x)dx}{w(Q)}\right)} \le t_0=C_{n,p}A,
\end{equation*}
and the desired estimate follows.

\textbf{Case 2:}
\begin{equation*}
    \sup_{0<t\le t_0} t \left(\frac{w(\{x\in Q:|f(x)-f_Q|>t\})}{w(Q)}\right)^\frac{1}{p^*} \le \sup_{t_0\le t} t \left(\frac{w(\{x\in Q:|f(x)-f_Q|>t\})}{w(Q)}\right)^\frac{1}{p^*}.
\end{equation*}
If the last supremum is zero, then the desired estimate is immediate. Otherwise, by the definition of the supremum, there exists $t_1\ge t_0$ such that
\begin{align*}
    \frac{t_1}{2} \left( \frac{w\left(\left\{x\in Q:|f(x)-f_Q|>\frac{t_1}{2}\right\}\right)}{w(Q)} \right)^\frac{1}{p^*} 
    &\le \left\| f-f_Q\right\|_{\strt{2.3ex} L^{p^*,\infty}\left(Q,\frac{w(x)dx}{w(Q)}\right)} \\
    &= \sup_{t_0\le t} t \left(\frac{w(\{x\in Q:|f(x)-f_Q|>t\})}{w(Q)}\right)^\frac{1}{p^*} \\
    &\le 2t_1 \left( \frac{w\left(\left\{x\in Q:|f(x)-f_Q|>t_1\right\}\right)}{w(Q)} \right)^\frac{1}{p^*}.
\end{align*}
This implies that 
\begin{equation}\label{level comparison fQ}
    w\left(\left\{x\in Q:|f(x)-f_Q|>\frac{t_1}{2}\right\}\right) \le c_{n,p} w\left(\left\{x\in Q:|f(x)-f_Q|>t_1\right\}\right).
\end{equation}
Consider the set defined by 
\begin{equation*}
    E_{t_1}:=\{x\in Q:|f(x)-f_Q|>t_1\}.
\end{equation*}
We proceed by considering the truncation operator
$\tau_\lambda$ defined as follows:
\begin{equation*}
    \tau_\lambda(g)(x):=\begin{cases}0, & \text{if } g(x)\le \lambda,\\ g(x)-\lambda, & \text{if } \lambda<g(x)\le 2\lambda,\\ \lambda, & \text{if } g(x)>2\lambda.\end{cases}
\end{equation*}
Note that, for each $x\in E_{t_1}$, we have
\begin{equation*}
    \frac{t_1}{2}=\tau_{\frac{t_1}{2}}(|f-f_Q|)(x) \le \left| \tau_{\frac{t_1}{2}}(|f-f_Q|)(x)-\left(\tau_{\frac{t_1}{2}}(|f-f_Q|)\right)_Q \right|+\left(\tau_{\frac{t_1}{2}}(|f-f_Q|)\right)_Q.
\end{equation*}
Using the fact that $\tau_\lambda(g)(x)\le g(x)$ and \eqref{I1 P11}, we obtain
\begin{align*}
    \left(\tau_{\frac{t_1}{2}}(|f-f_Q|)\right)_Q
    &\le \frac{1}{|Q|}\int_Q |f(x)-f_Q|dx \\
    &\le c_n \ell(Q) \frac{1}{|Q|} \int_Q |\nabla f(x)|dx \\
    &= c_n \frac{1}{|Q|^\frac{1}{n'}} \int_Q |\nabla f(x)|dx.
\end{align*}
We now insert the weight in the previous estimate. Since
\begin{equation*}
    \frac{w(Q)}{|Q|} \le \inf_{x\in Q} M(w\chi_Q)(x),
\end{equation*}
we have
\begin{align*}
    \left(\tau_{\frac{t_1}{2}}(|f-f_Q|)\right)_Q
    &\le c_n \frac{1}{w(Q)^\frac{1}{n'}} \inf_{x\in Q} M(w\chi_Q)(x)^\frac{1}{n'} \int_Q |\nabla f(x)|dx \\
    &\le c_n \frac{1}{w(Q)^\frac{1}{n'}} \int_Q |\nabla f(x)|M(w\chi_Q)(x)^\frac{1}{n'}dx \\
    &= c_n w(Q)^\frac{1}{n} \left\| \left| \nabla f \right| \frac{M(w\chi_Q)^\frac{1}{n'}}{w} \right\|_{\strt{2.3ex} L^1\left(Q,\frac{w(x)dx}{w(Q)}\right)} \\
    &\le c_{n,p} w(Q)^\frac{1}{n} \left\| \left| \nabla f \right| \frac{M(w\chi_Q)^\frac{1}{n'}}{w} \right\|_{\strt{2.3ex} L^{p,\infty}\left(Q,\frac{w(x)dx}{w(Q)}\right)} \\
    &= c_{n,p}A.
\end{align*}
In the last inequality, we used Kolmogorov's inequality again. Choosing the constant $C_{n,p} = 4c_{n,p}$ in the definition of $t_0$, and recalling that $t_1\ge t_0$, we obtain
\begin{equation*}
    \left(\tau_{\frac{t_1}{2}}(|f-f_Q|)\right)_Q \le \frac{t_1}{4}.
\end{equation*}
As a result,
\begin{equation*}
    E_{t_1} \subset \left\{x\in Q:\left| \tau_{\frac{t_1}{2}}(|f-f_Q|)(x)-\left(\tau_{\frac{t_1}{2}}(|f-f_Q|)\right)_Q \right| \ge \frac{t_1}{4}\right\}.
\end{equation*}

Applying Theorem \ref{Thm FPW} to $\tau_{\frac{t_1}{2}}(|f-f_Q|)$ with the measure $w(x)dx$, we get
\begin{align*}
    \frac{t_1}{4}w(E_{t_1})^\frac{1}{n'}
    &\le \frac{t_1}{4} w\left( \left\{ x\in Q: \left| \tau_{\frac{t_1}{2}}(|f-f_Q|)(x) - \left(\tau_{\frac{t_1}{2}}(|f-f_Q|)\right)_Q \right| \ge \frac{t_1}{4} \right\} \right)^\frac{1}{n'} \\
    &\le \left\| \tau_{\frac{t_1}{2}}(|f-f_Q|) - \left(\tau_{\frac{t_1}{2}}(|f-f_Q|)\right)_Q \right\|_{\strt{2.3ex} L^{n',\infty}(Q, w)} \\
    &\le c_n \int_Q \left|\nabla \big( \tau_{\frac{t_1}{2}}(|f-f_Q|)\big)(x)\right| M^c(w\chi_Q)(x)^\frac{1}{n'}dx \\
    &\le c_n \int_{\left\{x\in Q:\frac{t_1}{2}<|f(x)-f_Q|<t_1\right\}} |\nabla f(x)| \frac{M(w\chi_Q)(x)^\frac{1}{n'}}{w(x)} w(x)dx.
\end{align*}
Using the Lorentz-space version of Hölder's inequality \eqref{Holder Lorentz}, it follows that, 
\begin{equation*}
    \frac{t_1}{4}w(E_{t_1})^\frac{1}{n'} \le c_{n,p} \left\| \chi_{\left\{x\in Q:\frac{t_1}{2}<|f(x)-f_Q|<t_1\right\}} \right\|_{\strt{2.3ex}L^{p',1}(Q,w)} \left\| \left| \nabla f \right| \frac{M(w\chi_Q)^\frac{1}{n'}}{w} \right\|_{\strt{2.3ex}L^{p,\infty}(Q,w)}.
\end{equation*}
Moreover, by \eqref{level comparison fQ}, 
\begin{align*}
    \left\| \chi_{\left\{x\in Q:\frac{t_1}{2}<|f(x)-f_Q|<t_1\right\}} \right\|_{\strt{2.3ex} L^{p',1}(Q,w)}
    &= c_p w\left(\left\{x\in Q:\frac{t_1}{2}<|f(x)-f_Q|<t_1\right\}\right)^\frac{1}{p'} \\
    &\le c_{n,p} w\left(E_{t_1}\right)^\frac{1}{p'}.
\end{align*}
Therefore,
\begin{equation*}
    t_1 w(E_{t_1})^\frac{1}{n'} \le c_{n,p} w(E_{t_1})^\frac{1}{p'} \left\| \left| \nabla f \right| \frac{M(w\chi_Q)^\frac{1}{n'}}{w} \right\|_{\strt{2.3ex} L^{p,\infty}(Q,w)}.
\end{equation*}
Since $\frac{1}{n'}-\frac{1}{p'}=\frac{1}{p^*}$, we obtain
\begin{equation}\label{key level estimate fQ}
    t_1 w(E_{t_1})^\frac{1}{p^*} \le c_{n,p} \left\| \left| \nabla f \right| \frac{M(w\chi_Q)^\frac{1}{n'}}{w} \right\|_{\strt{2.3ex} L^{p,\infty}(Q,w)}.
\end{equation}
Combining everything together, we have
\begin{align*}
    \left\| f-f_Q\right\|_{L^{p^*,\infty}\left(Q,\frac{w(x)dx}{w(Q)}\right)}
    &\le 2t_1 \left(\frac{w(E_{t_1})}{w(Q)}\right)^\frac{1}{p^*} \\
    &\le c_{n,p} \left(\frac{1}{w(Q)}\right)^\frac{1}{p^*} \left\| \left| \nabla f \right| \frac{M(w\chi_Q)^\frac{1}{n'}}{w} \right\|_{\strt{2.3ex} L^{p,\infty}(Q,w)} \\
    &= c_{n,p} w(Q)^\frac{1}{n} \left\| \left| \nabla f \right| \frac{M(w\chi_Q)^\frac{1}{n'}}{w} \right\|_{\strt{2.3ex} L^{p,\infty}\left(Q,\frac{w(x)dx}{w(Q)}\right)} \\
    &= c_{n,p}A.
\end{align*}
This proves \eqref{Thm PS 4 normalized Eq}, and concludes the proof. 
\end{proof}

As announced after Theorem \ref{Thm PS 4}, we briefly indicate how Theorem \ref{Thm PS 4} implies a refinement of \cite[Theorem 1.21]{PR} in the case $p>1$. Using \eqref{Thm PS 4 Eq} and \eqref{eq:lorentz-nestedness}, we obtain
\begin{align*}
    \left\| f-f_Q\right\|_{\strt{2.3ex} L^{p^*,\infty}(Q,w)}
    &\le c_{n,p}\left\| \left| \nabla f \right| \frac{M^c(w\chi_Q)^\frac{1}{n'}}{w}\right\|_{\strt{2.3ex} L^{p,\infty}(Q,w)} \\
    &\le c_{n,p}\left(\int_Q |\nabla f(x)|^p \frac{M^c(w\chi_Q)(x)^\frac{p}{n'}}{w(x)^{p-1}}dx\right)^\frac{1}{p}.
\end{align*}
Moreover, following the argument used in the proof of Theorem \ref{Thm PS 4}, starting from the classical $L^1$ Poincaré inequality on cubes and inserting the weight as in that proof, we also get
\begin{equation*}
    \frac{1}{|Q|}\int_Q |f(x)-f_Q|\,dx
    \le c_{n,p} w(Q)^{-\frac{1}{p^*}}\left(\int_Q |\nabla f(x)|^p \frac{M^c(w\chi_Q)(x)^\frac{p}{n'}}{w(x)^{p-1}}dx\right)^\frac{1}{p}.
\end{equation*}
We now apply Theorem \ref{thm:truncation-method} with $r=p\le q=p^*$,
\begin{equation*}
    d\mu(x)=\frac{w(x)\,dx}{w(Q)}
\end{equation*}
and
\begin{equation*}
    d\nu(x)=w(Q)^{-\frac{p}{p^*}}\frac{M^c(w\chi_Q)(x)^\frac{p}{n'}}{w(x)^{p-1}}\,dx,
\end{equation*}
and we obtain
\begin{equation*}
    \left\| f-f_{Q}\right\|_{\strt{2.3ex} L^{p^*,p}(Q,\mu)}
    \le c_{n,p}\left(\int_Q |\nabla f(x)|^p\,d\nu(x)\right)^\frac{1}{p}.
\end{equation*}
Multiplying by $w(Q)^\frac{1}{p^*}$, we conclude that
\begin{align*}
    \left(\int_Q |f(x)-f_{Q}|^{p^*}w(x)\,dx\right)^\frac{1}{p^*}
    &\le \left\| f-f_{Q}\right\|_{\strt{2.3ex} L^{p^*,p}(Q,w)} \\
    &\le c_{n,p}\left(\int_Q |\nabla f(x)|^p \frac{M^c(w\chi_Q)(x)^\frac{p}{n'}}{w(x)^{p-1}}dx\right)^\frac{1}{p},
\end{align*}
which is a refinement of \cite[Theorem~1.21]{PR}, since the unweighted average $f_Q$ appears on the left-hand side.

\begin{proof}[Proof of Corollary \ref{Thm PS 5}]
Dividing by $w(Q)^\frac{1}{p^*}$ in \eqref{Thm PS 4 Eq} and using the definition of the Sobolev exponent $p^*$, we obtain
\begin{align*}
    \left\| f-f_Q\right\|_{\strt{2.3ex} L^{p^*,\infty}\left(Q,\frac{w(x)dx}{w(Q)}\right)}
    &\le c_{n,p} w(Q)^\frac{1}{n}\left\| \left| \nabla f \right| \frac{M^c(w\chi_Q)^\frac{1}{n'}}{w}\right\|_{\strt{2.3ex} L^{p,\infty}\left(Q,\frac{w(x)dx}{w(Q)}\right)} \\
    &= c_{n,p} \ell(Q)\left(\frac{w(Q)}{|Q|}\right)^\frac{1}{n}\left\| \left| \nabla f \right| \frac{M^c(w\chi_Q)^\frac{1}{n'}}{w}\right\|_{\strt{2.3ex} L^{p,\infty}\left(Q,\frac{w(x)dx}{w(Q)}\right)} \\
    &\le c_{n,p} \ell(Q)\inf_{x\in Q}\left(M(w\chi_Q)(x)\right)^\frac{1}{n}\left\| \left| \nabla f \right| \frac{M^c(w\chi_Q)^\frac{1}{n'}}{w}\right\|_{\strt{2.3ex} L^{p,\infty}\left(Q,\frac{w(x)dx}{w(Q)}\right)} \\
    &\le c_{n,p} \ell(Q)\left\| \left| \nabla f \right| \frac{M^c(w\chi_Q)}{w}\right\|_{\strt{2.3ex} L^{p,\infty}\left(Q,\frac{w(x)dx}{w(Q)}\right)} \\
    &\le c_{n,p} [w]_{A_1} \ell(Q)\left\| \nabla f \right\|_{\strt{2.3ex} L^{p,\infty}\left(Q,\frac{w(x)dx}{w(Q)}\right)},
\end{align*}
where in the third inequality we used the fact that $M$ and $M^c$ are pointwise equivalent, and in the last inequality we used that $w\in A_1$.
\end{proof}

\appendix
\section{Generalized Poincaré inequalities with polynomials}\label{S Polynomials}

Fix a cube $Q$ and an integer $m\in \N\cup \{0\}$. We consider the space $\mathcal{P}_m(Q)$ of polynomials of degree at most $m$ in $n$ variables restricted to the cube $Q$. We denote by $P_Q^m f$ the projection of a function $f$ onto the space $\mathcal{P}_m(Q)$.	We have the following property: there exists a constant $c_{n,m}>0$ such that 
	\begin{equation}\label{ProyPolinomio1}
		\left\| P_Q^m f \right\|_{\strt{2.3ex} L^\infty(Q)} \le c_{n,m} \frac{1}{|Q|} \int_Q |f(x)|  dx .
	\end{equation}
We refer to \cite[Section 8]{PR} for a detailed explanation of this setting. The polynomial framework here is an extension of the average we discussed in previous sections, as $P_Q^0 f = f_Q$.

We state the main result of this section.

\begin{theorem}\label{Thm Pol 1}
	Let $1 \le p < \infty$, $1 \le r < \infty$, $1\le s<\infty$, and let $a$ be a functional over cubes. Let $w\in A_r$ and let $f:\R^n \longrightarrow \R$ be a locally integrable function satisfying
	\begin{equation*}
		\frac{1}{|Q|}\int_Q |f(x)-P_Q^mf(x)|\,dx \le a(Q)
	\end{equation*}
	for every cube $Q$.
	
	\begin{itemize}
		\item Assume that $a\in SD_p^s(w)$ and $p<rs$. Then there exists a constant $c_{n,m}>0$ such that, for every cube $Q$,
		\begin{equation}\label{eq:mainpoly}
			\norm{f-P_Q^mf}_{\strt{2.3ex} L^{p_{r,s}^{*},\infty}\left(Q, \frac{w(x)\,dx}{w(Q)}\right)}
			\le
			c_{n,m} p_{r,s}^{*}\,[w]_{A_\infty}\,[w]_{A_r}^{\frac{1}{rs}}\,
			\|a\|\,a(Q),
		\end{equation}
		where $p_{r,s}^{*}$ is defined by the relation
		\begin{equation*}
			\frac{1}{p}-\frac{1}{p_{r,s}^{*}}=\frac{1}{s \, r}.
		\end{equation*}
		
		\item Assume that $a\in SD_p^s(w)$ and $p\ge rs$. Then there exists a constant $c_{n,m,p}>0$ such that, for every cube $Q$,
		\begin{equation}\label{eq:exppoly}
			\|f-P_Q^mf\|_{\exp L\left(Q,\frac{w(x)\,dx}{w(Q)}\right)}
			\le
			c_{n,m,p}\,[w]_{A_\infty}\,[w]_{A_r}^{\frac1p}\,
			\|a\|\,a(Q).
		\end{equation}
	\end{itemize}
\end{theorem}

\begin{proof}[Sketch of the proof]
We follow the proof of Theorem~\ref{Selfimproving Ar}, replacing averages by the projections $P_Q^m f$. Fix a cube $Q$ and set $\Omega_t:=\{x\in Q:\ M_Q(f-P_Q^m f)(x)>t\}$. Let $\{Q_j\}_j$ be the maximal dyadic subcubes forming $\Omega_t$. Then either $\Omega_t=Q$, or
\begin{equation*}
    t<\frac{1}{|Q_j|}\int_{Q_j}|f-P_Q^m f|\,dx\le 2^n t.
\end{equation*}
Using $P_{Q_j}^m f-P_Q^m f=P_{Q_j}^m(f-P_Q^m f)$ and \eqref{ProyPolinomio1}, one gets $\|P_{Q_j}^m f-P_Q^m f\|_{L^\infty(Q_j)}\le c_m 2^n  t$, and hence for a suitable $\kappa=\kappa(n,m)>1$,
\begin{equation*}
    w(\Omega_{\kappa t})\le \sum_j w(E_{Q_j}), \qquad E_{Q_j}:=\{x\in Q_j:\ M_{Q_j}(f-P_{Q_j}^m f)(x)>t\}.
\end{equation*}
Introduce the polynomial sharp maximal function
\begin{equation*}
    M_Q^{\#,m}f(x):=\sup_{x\in R\in\mathcal D(Q)}\frac{1}{|R|}\int_R |f-P_R^m f|.
\end{equation*}
For $\gamma>0$ split each $E_{Q_j}$ into the regions where $M_Q^{\#,m}f\le \gamma t$ and where $M_Q^{\#,m}f>\gamma t$. The polynomial good-lambda inequality (obtained from \cite[Theorem~7.1]{CP}) yields
\begin{equation*}
    w(\Omega_{\kappa t})\le c_1 e^{-\frac{c_2}{[w]_{A_\infty}\gamma}}\,w(\Omega_t) + w(\{x\in Q : M_Q^{\#,m}f(x)>\gamma t\}).
\end{equation*}
The second term is estimated exactly as in the proof of \eqref{Thm Ar eq} using the hypothesis $a\in SD_p^s(w)$. The exponential estimate \eqref{eq:exppoly} follows from \eqref{eq:mainpoly} via Proposition \ref{prop:linearweakLptoexpL}.
\end{proof}

As a consequence of this result, and in view of the lack of a truncation method for higher-order derivatives, we obtain the following analogue of Corollary \ref{thm:PS-Lorentz-weak}.

\begin{corollary}\label{cor:PS-Lorentz-weak-poly}
Let $1\le m <n$. Let $w\in A_r$ with $1\le r\le p$, $1< p< \frac{n}{m}r$ and consider $p_{m,r}^{*}$ defined by
\begin{equation*}
    \frac{1}{p}-\frac{1}{p_{m,r}^{*}}= \frac{m}{n} \frac{1}{r}.
\end{equation*}
Then, there exists a constant $c_{n,m}>0$ such that 
\begin{equation*}
	\|f-P_Q^{m-1}f\|_{\strt{2.3ex} L^{p_{m,r}^{*},\infty}\left(Q, \frac{w(x)\,dx}{w(Q)}\right)} \le c_{n,m} p^*_{m,r} [w]_{A_\infty}\,[w]_{A_{p,p_{m,r}^{*}}}^{\frac{1}{p}}\,[w]_{A_r}^{\frac{m}{rn}}\, \ell(Q)^m \,\|\nabla^m f\|_{L^{p,p_{m,r}^{*}}\left(Q,\frac{w(x)\,dx}{w(Q)}\right)},
\end{equation*}
for every cube $Q$. 
\end{corollary}

The proof of this result is essentially the same as that of Corollary \ref{thm:PS-Lorentz-weak}, but using the following starting inequality from \cite{B}: there exists a constant $C>1$ such that
\begin{equation*}
	\frac{1}{|Q|}\int_Q |f(x)-P_Q^{m-1} f(x)| dx\le C \ell(Q)^{m} \frac{1}{|Q|} \int_Q |\nabla^{m} f (x)|dx ,
\end{equation*}
	for every cube $Q$. The rest of the proof follows from the preceding arguments; we omit it here.

\section{Self-improving of vector-valued Poincaré inequalities}\label{S vector}

In this appendix, we extend the main self-improving result of the paper to the vector-valued setting, where the oscillation is measured through the $\ell_q$-norm.

\begin{theorem}{\label{Appendix B}}
	Let $q\ge 1$, $1 \le p < \infty$, $1 \le r < \infty$, $1\le s<\infty$, and let $a$ be a functional over cubes. Let $w\in A_r$ and let $f:\R^n \longrightarrow \ell_q$ be a locally integrable function satisfying
	\begin{equation*}
		\frac{1}{|Q|}\int_Q \norm{f(x)-f_Q}_{\ell_q}\,dx \le a(Q)
	\end{equation*}
	for every cube $Q$.
	
	\begin{itemize}
		\item Assume that $a\in SD_p^s(w)$ and $p<rs$. Then there exists a dimensional constant $c_n>0$ such that, for every cube $Q$,
		\begin{equation*}
			\norm{\norm{f-f_Q}_{\ell_q}}_{\strt{2.3ex} L^{p_{r,s}^{*},\infty}\left(Q, \frac{w(x)\,dx}{w(Q)}\right)}
			\le
			c_n p_{r,s}^{*}\,[w]_{A_\infty}\,[w]_{A_r}^{\frac{1}{rs}}\,
			\|a\|\,a(Q),
		\end{equation*}
		where $p_{r,s}^{*}$ is defined by the relation
		\begin{equation*}
			\frac{1}{p}-\frac{1}{p_{r,s}^{*}}=\frac{1}{s\,r}.
		\end{equation*}
		
		\item Assume that $a\in SD_p^s(w)$ and $p\ge rs$. Then there exists a constant $c_{n,p}>0$ such that, for every cube $Q$,
		\begin{equation*}
			\norm{\norm{f-f_Q}_{\ell_q}}_{\exp L\left(Q,\frac{w(x)\,dx}{w(Q)}\right)}
			\le
			c_{n,p}\,[w]_{A_\infty}\,[w]_{A_r}^{\frac1p}\,
			\|a\|\,a(Q).
		\end{equation*}
	\end{itemize}
\end{theorem}

\begin{proof}[Sketch of the proof]
This is a direct vector-valued adaptation of the scalar proof: one applies the whole argument to the scalar function $x\mapsto \|f(x)-f_Q\|_{\ell_q}$. The Calderón--Zygmund decomposition and the key estimate $w(\Omega_{\kappa t})\le \sum_j w(E_{Q_j})$ follow from the triangle inequality in $\ell_q$ and Minkowski's inequality. The good-lambda step applies to the corresponding sharp maximal function built from $\|f-f_R\|_{\ell_q}$, and the remaining estimate uses $a\in SD_p^s(w)$ exactly as in the scalar case. The exponential endpoint follows from the weak bounds via Proposition~\ref{prop:linearweakLptoexpL}.
\end{proof}

We also remark that one can prove a vector-valued version of \cite[Theorem 1.1]{CP}. Such a vector-valued extension would yield the corresponding good-lambda inequality needed in the previous argument.

In order to derive concrete consequences of this vector-valued self-improving result, we shall use the following lemma from \cite{MP}.

\begin{lemma}
	Let $q\ge 1$. There exists a dimensional constant $c_n>0$ such that, for any cube $Q$ of $\R^n$ and for any vector-valued function $f: \R^n\longrightarrow \ell_q$ with components in $C^1(Q)$, the representation formula 
	\begin{equation*}
		\left\| f(x)-f_Q \right\| _{\ell_q} \le c_n I_1(\left\| \nabla f\right\|_{\ell_q} \chi_Q) (x) 
	\end{equation*}
	holds for every $x\in Q$, where $I_1$ is the Riesz potential defined by $I_1g(z):=\int_{\mathbb{R}^n}\frac{g(y)}{|z-y|^{n-1}}dy$. 
\end{lemma}

Since $I_1$ is self-adjoint, the previous result implies the following vector-valued $(1,1)$ Poincaré inequality.

\begin{proposition}
	Let $q\ge 1$. There exists a dimensional constant $c_n>0$ such that, for any cube $Q$ of $\R^n$ and for any vector-valued function $f: \R^n\longrightarrow \ell_q$ with components in $C^1(Q)$, we have 
	\begin{equation*}
		\frac{1}{|Q|} \int_Q \left\| f(x)-f_Q \right\| _{\ell_q} dx \le c_n \ell(Q) \frac{1}{|Q|} \int_Q \left\| \nabla f(x)\right\|_{\ell_q}  dx .
	\end{equation*} 
\end{proposition}

Consequently, noting the fact that the truncation method is not known to work in the vector-valued setting (as noted in \cite[p. 2]{MP}), we obtain the following analogue of Corollary \ref{thm:PS-Lorentz-weak}.
 
\begin{corollary}
	Let $q\ge 1$. Let $w\in A_r$ with $1\le r\le p$, $1< p< nr$ and consider $p_{r}^{*}$ defined by
\begin{equation*}
    \frac{1}{p}-\frac{1}{p_{r}^{*}}= \frac{1}{n} \frac{1}{r}. 
\end{equation*}
Then, there exists a dimensional constant $c_n>0$ such that
	\begin{equation*}
		\norm{\norm{f-f_Q}_{\ell_q}}_{\strt{2.3ex} L^{p^*_{r},\infty}\left(Q, \frac{w dx}{w(Q)}\right)}
		\le c_n p^*_{r} [w]_{A_\infty}\,[w]_{A_{p, p^*_{r}}}^{\frac{1}{p}}\,[w]_{A_r}^{\frac{1}{rn}}\ell(Q) \norm{ \|\nabla f\|_{\ell_q}}_{\strt{2.3ex} L^{p, p^*_{r}}\left(Q, \frac{w dx}{w(Q)}\right)},
	\end{equation*}
    for every cube $Q$.
\end{corollary}

\section{Further extensions: Rectangles}\label{S Rectangles}

In this appendix, we generalize Theorems \ref{Selfimproving Ar} and \ref{Thm Pol 1} to the context of rectangles in $\R^n$. We borrow the terminology from  \cite{CMPR}. We denote by $\mathcal{R}$ the family of rectangles in $\R^n$. By a rectangle, we mean the product of $n$ intervals in $\R$.  We will denote by $SD_{p,\mathcal{R}}^s(w)$ the condition analogous to the $SD_p^s(w)$ condition, considering rectangles instead of cubes. We will state the result in the context of polynomial approximation, with the same definitions as in the previous section. 

\begin{theorem}\label{Thm Rectangles}
	Let $1 \le p < \infty$, $1 \le r < \infty$, $1\le s<\infty$, and let $a$ be a functional over rectangles. Let $w\in A_{r,\mathcal{R}}$ and let $f:\R^n \longrightarrow \R$ be a locally integrable function satisfying
	\begin{equation*}
		\frac{1}{|R|}\int_R |f(x)-P_R^mf(x)|\,dx \le a(R)
	\end{equation*}
	for every rectangle $R$.
	
	\begin{itemize}
		\item Assume that $a\in SD_{p,\mathcal{R}}^s(w)$ and $p<rs$. Then there exists a constant $c_{n,m}>0$ such that, for every rectangle $R$,
		\begin{equation*}
			\norm{f-P_R^mf}_{\strt{2.3ex} L^{p_{r,s}^{*},\infty}\left(R, \frac{w(x)\,dx}{w(R)}\right)}
			\le
			c_{n,m} p_{r,s}^{*}\,[w]_{A_{\infty,\mathcal{R}}}\,[w]_{A_{r,\mathcal{R}}}^{\frac{1}{rs}}\,
			\|a\|\,a(R),
		\end{equation*}
		where $p_{r,s}^{*}$ is defined by the relation
		\begin{equation*}
			\frac{1}{p}-\frac{1}{p_{r,s}^{*}}=\frac{1}{s\,r}.
		\end{equation*}
		
		\item Assume that $a\in SD_{p,\mathcal{R}}^s(w)$ and $p\ge rs$. Then there exists a constant $c_{n,m,p}>0$ such that, for every rectangle $R$,
		\begin{equation*}
			\|f-P_R^mf\|_{\exp L\left(R,\frac{w(x)\,dx}{w(R)}\right)}
			\le
			c_{n, m,p}\,[w]_{A_{\infty,\mathcal{R}}}\,[w]_{A_{r,\mathcal{R}}}^{\frac1p}\,
			\|a\|\,a(R).
		\end{equation*}
	\end{itemize}
\end{theorem}

The proof is the same as that of Theorem \ref{Thm Pol 1} but using the Rising Sun Lemma instead of Calderón--Zygmund decomposition. We state here the lemma for the sake of completeness.
		
	\begin{lemma}[\cite{KLS}]
		Let $R$ be a rectangle in $\R^n$ and let $f$ be a function such that 
		\begin{equation*}
			\frac{1}{|R|}\int_R f(x)dx \le A.
		\end{equation*}
		Then, there exists a pairwise disjoint family of rectangles $\{R_j\}_j$ such that
		\begin{itemize}
			\item $R_j\subset R$ for all $j$.
			\item $\frac{1}{|R_j|}\int_{R_j} f(x)dx = A$ for all $j$.
			\item $f(x)\le A$ a.e. $x\in R\setminus \bigcup_j R_j$. 
		\end{itemize}
	\end{lemma} 

	\begin{remark}
	The preceding lemma does not introduce any dimensional constant. The dimensional dependence in the previous result comes instead from the polynomial projection estimate \eqref{ProyPolinomio1}.  In the non-polynomial case, that is, when the oscillation is measured with respect to the average $f_Q$, Theorem \ref{Thm Rectangles} does not require any dimensional constants.
\end{remark}

\section*{Acknowledgments}
 The authors would like to thank Andrea Cianchi for valuable discussions concerning Theorem \ref{Thm PS 4} during his visit to BCAM in May 2025. The authors also thank Emiel Lorist for helpful comments on quantitative aspects of weighted Poincaré--Sobolev inequalities.

\section*{Conflicts of Interest}
The authors have no conflicts of interest to declare.

\section*{Data Availability Statement}
Data sharing is not applicable to this article, as no datasets were generated or analyzed during the current study.

\bibliographystyle{amsplain}

\end{document}